\documentclass[a4paper,11pt,leqno]{article}

\usepackage[latin1]{inputenc}
\usepackage[T1]{fontenc} %only using latex
\usepackage[english]{babel}
\usepackage{verbatim}
\usepackage{calligra}
\usepackage{enumerate}
\usepackage{dsfont}
\usepackage{amsfonts}
\usepackage{amsmath}
\usepackage{amsthm}
\usepackage{amssymb}
\usepackage{fancyhdr}
\usepackage{mathtools}
\usepackage{mathrsfs}
\usepackage{color}
\usepackage{graphicx}
\usepackage{bm}

\theoremstyle{definition}
\newtheorem{defin}{Definition}[section]

\newtheorem{rem}[defin]{Remark}

\theoremstyle{plane}
\newtheorem{thm}[defin]{Theorem}
\newtheorem{prop}[defin]{Proposition}
\newtheorem{coroll}[defin]{Corollary}
\newtheorem{lemma}[defin]{Lemma}

\newcommand{\mc}{\mathcal}
\newcommand{\mf}{\mathfrak}
\newcommand{\veps}{\varepsilon}
\newcommand{\what}{\widehat}
\newcommand{\wtilde}{\widetilde}
\newcommand{\vphi}{\varphi}
\newcommand{\oline}{\overline}

\newcommand{\ra}{\rightarrow}

\newcommand{\hra}{\hookrightarrow}

\newcommand{\g}{\gamma}
\newcommand{\s}{\sigma}
\newcommand{\z}{\zeta}
\newcommand{\lam}{\lambda}
\newcommand{\de}{\delta}
\newcommand{\lan}{\langle}
\newcommand{\ran}{\rangle}

\newcommand{\R}{\mathbb{R}}

\newcommand{\N}{\mathbb{N}}
\newcommand{\Z}{\mathbb{Z}}
\newcommand{\T}{\mathbb{T}}

\renewcommand{\div}{{\rm div}\,}

\newcommand{\curl}{{\rm curl}\,}

\newcommand{\Id}{{\rm Id}\,}

\allowdisplaybreaks

\def\d{\partial}
\def\div{{\rm div}\,}

\title{\textsc{\Large{\textbf{Asymptotics of fast rotating  density-dependent incompressible fluids in two space dimensions}}}}

\author{\normalsize\textsl{Francesco Fanelli}$\,^1\qquad$ and $\qquad$
\textsl{Isabelle Gallagher}$\,^{2}$ \vspace{.5cm} \\
\footnotesize{$\,^1\;$ \textsc{Universit\'e de Lyon, Universit\'e Claude Bernard Lyon 1}} \\
{\footnotesize \it Institut Camille Jordan -- UMR 5208} \\
\footnotesize{\ttfamily{fanelli@math.univ-lyon1.fr}} \vspace{.3cm} \\
\footnotesize{$\,^2\;$ \textsc{Universit\'e Paris Diderot, Sorbonne Paris Cit\'e}} \\
{\footnotesize \it Institut de Math\'ematiques de Jussieu-Paris Rive Gauche -- UMR 7586} \\
\footnotesize{\ttfamily{gallagher@math.univ-paris-diderot.fr}}
}

\vspace{.2cm}
\date\today

\begin{document}

\maketitle

%%%%%%%%%%%%%%%%%%%%%%%%%%%%%%%%%%%%%%%%%%%%%%%%%%%%%%%%%%%%%%%%%%%%%%%%%%%%%%%%%%%%%%%
\subsubsection*{Abstract}
{\footnotesize In the present paper we study the fast rotation limit for viscous incompressible fluids with variable density, whose motion is influenced  by the Coriolis force.
We restrict our analysis to two dimensional flows. In the case when the initial density is a small perturbation of a constant state, we recover in the limit the convergence to the homogeneous incompressible
Navier-Stokes equations (up to an additional term, due to density fluctuations). For general non-homogeneous fluids, the limit equations are instead linear, and
the limit dynamics is described in terms of the vorticity and  the density oscillation function: we lack enough regularity on the latter to prove convergence on the momentum equation itself.
The proof of both results relies on a compensated compactness argument, which enables one  to treat also the possible presence of vacuum.}

\paragraph*{2010 Mathematics Subject Classification:}{\small 35Q35 % PDEs / Eq. of math. phys. / PDEs in connection with fluid mechanics
(primary);
35B25, % PDEs / Qualitative properties / Singular perturbations
76U05, % Fluid Mechanics / Rotating fluids / Rotating fluids
35Q86, % PDEs / Equations of mathematical physics and other areas of application / PDEs in connection with geophysics
35B40, % PDEs / Qualitative properties / Asymptotic behavior of solutions
76M45 % Fluid Mechanics / Basic methods in Fluid Mechanics / Asymptotic methods, singular perturbations
(secondary).}

\paragraph*{Keywords:}{\small incompressible fluids; Navier-Stokes equations; variable density; vacuum; Coriolis force; singular perturbation problem; low Rossby number.}

%%%%%%%%%%%%%%%%%%%%%%%%%%%%%%%%%%%%%%%%%%%%%%%%%%
%%%%%%%%%%%%%%%%%%%%%%%%%%%%%%%%%%%%%%%%%%%%%%%%%%
\section{Introduction} \label{s:intro}
%%%%%%%%%%%%%%%%%%%%%%%%%%%%%%%%%%%%%%%%%%%%%%%%%%
%%%%%%%%%%%%%%%%%%%%%%%%%%%%%%%%%%%%%%%%%%%%%%%%%%

Given a small parameter $\veps\,\in]0,1]$, let us consider the non-homogeneous incompressible Navier-Stokes-Coriolis system
\begin{equation} \label{eq:i_dd-NSC}
\left\{\begin{array}{l}
        \d_t\rho\,+\,\mbox{div} \,(\rho   u)\,=\,0 \\[1ex]
        \d_t(\rho\,u)\,+\,\mbox{div}\,(\rho\,u\otimes   u)\,+\,\dfrac{1}{\veps}\,\nabla\Pi\,+\,\dfrac{1}{\veps}\,\mf{C}(\rho,u)\,-\,
\nu\,\Delta u\,=\,0 \\[1ex]
\div u\,=\,0
       \end{array}
\right.
\end{equation}
in a domain $\Omega\subset\R^d$, with $d=2$ or $3$.  
This system   describes the dynamics of a viscous incompressible fluid with variable density, and whose dynamics is
  influenced by the rotation of the ambient physical system: our main motivation concerns ocean currents on the Earth surface, but the system could also relate to flows on stars or other rotating celestial bodies.
The scalar function~$\rho\geq0$ represents the density of the fluid, $u\in\R^d$ its velocity field and $\Pi$ its pressure; the term $\nabla\Pi$ can be
interpreted as a Lagrangian multiplier associated with the incompressibility constraint $\div u=0$. The parameter $\nu>0$ is the viscosity coefficient, and we have chosen
the very simple form $\Delta u$ for the viscous stress tensor of the fluid. Finally, we have denoted by $\mf{C}(\rho,u)$ the Coriolis operator, which   represents
the influence of the rotation on the fluid motion.

Operator $\mf C$ can assume various forms, depending on the precise assumptions and features one is interested in. In this paper we   place ourselves at mid-latitudes,
that is to say in a not too extended region of the Earth surface which is far enough from the poles and the equator, and we   suppose the rotation axis to be constant. These
are perhaps quite restrictive assumptions from the physical viewpoint, but they are usually assumed in mathematical studies (see e.g. \cite{C-D-G-G}):
indeed, the  model one obtains is already able to describe several important features of ocean dynamics (for instance the Taylor-Proudman theorem, the formation of boundary layers...).
Of course, more precise models can be considered, see e.g.  \cite{Des-Gre} and \cite{G-SR_2006}, where inhomogeneities of the Coriolis force are taken into account, or \cite{Dut}, 
 \cite{Dut-Maj}, \cite{D-M-S}, \cite{G_2008} and \cite{G-SR_Mem}  for an analysis of the~$\beta$-plane model for equatorial flows (where the Coriolis force vanishes).
In the present paper, we restrict our attention to the case $d=2$, for which the Coriolis operator takes the form
$$
\mf{C}(\rho,u)\,=\,\rho\,u^\perp\,,
$$
where, for any vector $v=(v^1,v^2)\in\R^2$, we have defined $v^\perp:=(-v^2,v^1)$.
After adimensionalization, this term is multiplied by a factor $1/{\rm Ro}$, where the \emph{Rossby number} ${\rm Ro}$ represents the inverse of the speed of the Earth rotation.
Hence, the scaling introduced in \eqref{eq:i_dd-NSC} corresponds to taking~${\rm Ro}=\veps$: our main goal here is to study the asymptotic behaviour of a family
$\bigl(\rho_\veps,u_\veps\bigr)_\veps$ of solutions to this system in the limit $\veps\ra0$ and to characterize the limit dynamics. The presence of the factor $1/\veps$ in front of
the pressure term comes from the remark that, in the limit of fast rotation (i.e. for $\veps$ going to $0$), the Coriolis term can be balanced only by $\nabla\Pi$, and so they
must be of the same order (see~\cite{C-D-G-G} and Subsection \ref{ss:constraints} for more details about this point).

The incompressibility constraint is well justified in a first approximation for oceanic flows, see e.g. books \cite{CR}, \cite{Gill} and \cite{Ped}.
In the case of constant density, i.e. when (say) $\rho\equiv1$ in \eqref{eq:i_dd-NSC}, the fast rotation limit has been deeply investigated in the last two decades. We refer e.g.
to \cite{B-M-N_1996}, \cite{B-M-N_1997}, \cite{B-M-N_1999}, \cite{C-D-G-G_2002}, \cite{G_1998}, \cite{G-SR_2006}, among various contributions (the present list is far from being exhaustive).
We refer to book \cite{C-D-G-G} for an overview of the results and further references. It goes without saying that, both for modeling and application purposes, it is important to include
in the model also other features of the fluid, like e.g. stratification, temperature variations, salinity\ldots
For systems allowing density fluctuations, mathematical studies are more recent, and to the best of our knowledge they concern only \emph{compressible flows}: see e.g.
\cite{B-D_2003},  \cite{G_2008}, \cite{G-SR_Mem} for $2$-D viscous shallow water models, where the depth function plays the same role as the density (we refer to  \cite{Dut}, \cite{Dut-Maj},
\cite{D-M-S} for the inviscid case), see \cite{F-G-GV-N}, \cite{F-G-N} for $3$-D barotropic Navier-Stokes equations (see also \cite{F_MA}, \cite{F_JMFM} for the case when strong surface tension is taken into
consideration). As is well-known, in the compressible (and barotropic) case the pressure is a given function of the density: then suitable scalings (roughly speaking, taking the Mach number of
order higher than or equal to that of the Rossby number, with respect to the small parameter $\veps$) enable one to find:
\begin{itemize}
 \item on the one hand, good uniform bounds (by energy estimates) for the density fluctuation functions;
\item on the other hand, a stream-function relation (by the analysis of the singular perturbation operator) linking the limit velocity and density profiles.
\end{itemize}
The combination of these two features is fundamental in the mathematical study, in order to handle the more difficult issues of the analysis (above all, at the convergence
level). In the end, in the limit one obtains a quasi-geostrophic type equation for the limit of the density fluctuations, and this is enough to
describe the whole asymptotic dynamics.

\medbreak
In the present paper we tackle instead the case of density-dependent fluids which are \emph{incompressible}, which is (to the best of our knowledge) completely new in the mathematical literature.
As a first approach, we restrict ourselves to the case $\Omega=\R^2$ or $\T^2$: indeed, on the one hand these simple geometries allow us to neglect boundary layers effects; on the other hand, already in
space dimension $d=2$, the analysis presents several difficulties. We expect this study to be a first step in the understanding of the problem in its whole generality for three-dimensional domains.
We   consider two different frameworks: the former is the \emph{slightly non-homogeneous case}, i.e. when the initial density is a small variation, of order $\veps$,
of a constant reference state; the latter is referred to as the \emph{fully non-homogeneous case}, because the initial density is a perturbation of an arbitrary positive state.
Let us remark that, in the latter instance, we are able to handle also the presence of possible vacuum.
%For convenience, we   mostly work away from vacuum: our emphasis here is on   handling density variations, and the arguments are somehow easier to present if the density is bounded
%away from $0$; we will explain at the end of the paper (see Subsection \ref{ss:vacuum}) the main modifications needed to treat vanishing densities.

At this point, let us make a physical \textsl{intermezzo} (see Chapter 3 of \cite{CR} or Appendix 3 of \cite{Gill} for details).
In the ocean, the water density is in general a complicated function of the pressure, temperature and salinity. Nonetheless, in many applications one usually neglects dependence on pressure
and assumes linear dependence on the other two quantities. Moreover, in a first approximation it is usual to suppose that temperature and salinity evolve through
a pure diffusion process (parabolic equations for each of these quantities).
A final simplification, commonly adopted in physical models (see e.g. Sections 2.6 and 2.9 of \cite{Les} or Chapter 1 of \cite{Maj}), is the so-called \emph{Boussinesq approximation}: namely,
the non-homogeneity of the fluid is supposed to be a small variation of a constant state $\oline\rho$ (without loss of generality, we can set $\oline{\rho}=1$).
In other words, one can write $\rho\,=\,1+\rho'$, where $|\rho'| \ll1$: then the incompressibility constraint is directly derived from the mass conservation equation at the highest order.
In addition, combining the equations for temperature (or energy) and salinity supplies   an explicit evolution equation, of parabolic type, for the density variation function $\rho'$.
We refer to Section 3.7 of \cite{CR} for more details.

In the light of the previous discussion, we notice here that the framework we adopt in this paper is ``critical'', in the following sense:
\begin{enumerate}[(i)]
\item first of all, the fluid is assumed to be incompressible, so that the pressure term is just a Lagrangian multiplier and   does not provide     any information on the density,
unlike   the case of  compressible fluids;
 \item moreover our model, although being simplified since it does not take into account other quantities like energy and salinity, lacks   information about the density fluctuations
(i.e. the function $\rho'$ mentioned above); even more, we do not want to restrict our attention only to small variations of constant states.
\end{enumerate}
Thus our study of the density must rely on the mass conservation relation only. In addition, this is a hyperbolic type equation (pure transport by the velocity field), for which we can
expect no gain of regularity: this fact is another source of difficulty in the analysis.

Nonetheless, the previous physical considerations suggest that the slightly non-homogeneous situation is somehow easier to be handled, and this turns out to be indeed the case.
As a matter of fact, since the constant state is simply transported by a divergence-free velocity field,   the density can be written as~$\rho_\veps=1+\rho_\veps'=1+\veps r_\veps$,
where, by pure transport again,  good   bounds on the variations $r_\veps$ can be derived. Then, an adaptation of the arguments
of \cite{G-SR_2006} enables us to pass to the limit directly in the mass and momentum equations: the limit dynamics is characterized by a coupling of a
two-dimensional incompressible homogeneous Navier-Stokes system for the limit velocity field $u$, with a transport equation for the limit fluctuation function $r$ by $u$ itself.
We point out that, in the momentum equation, an additional term $ru^\perp$ appears: it can be interpreted as a remainder of the oscillations of the density. This fact can be seen by inserting
the ansatz~$\rho_\veps=1+\veps r_\veps$ into the term $\mf{C}(\rho_\veps,u_\veps)/\veps$ and formally passing to the limit $\veps\ra0$ (in particular,
if $\rho_\veps\equiv1$ for all $\veps$, then   $r\equiv0$ and this term disappears).
The result about the slightly non-homogeneous case is given in Theorem \ref{t:hom}, and is proved in Sections \ref{s:sing-pert} and~\ref{s:hom}.

The fully non-homogeneous case is   more involved; the corresponding result is given in Theorem~ \ref{t:dens},     proved in Sections \ref{s:sing-pert} to~\ref{s:dens}. 
Keep in mind points (i)-(ii) mentioned above: since now the reference density is no longer constant, we   lack   information on the density fluctuations
(which, for the sake of clarity, will be called $\sigma_\veps$ in the fully non-homogeneous setting). This lack of information represents a major difficulty in our study;
in order to overcome it, we resort to the vorticity formulation of the momentum equation, and combine it with the mass equation.
This strategy provides  a link between the vorticity $\eta_\veps$ of
$v_\veps:=\rho_\veps u_\veps$ and $\sigma_\veps$, which nonetheless reveals to be of partial help only, since we derive from it bounds for $(\sigma_\veps)_\veps$ in very rough spaces.
Interpolating this information with another one coming from the mass equation   allows nevertheless to grasp a key property in order
to prove convergence in the general case (see Subsection \ref{ss:further}).
Strictly connected with this difficulty, another important issue in the analysis of the fully non-homogeneous framework is the control of oscillations, which is achieved
by studying the corresponding system of waves, which we call \emph{Rossby waves}. This terminology is typically used for fluctuations due to latitude (i.e., from the mathematical viewpoint,
due to variations of the rotation axis), but in our context having a truly non-constant reference density produces analogous effects. Since, roughly speaking, in this case the singular perturbation
operator has non-constant coefficients, in order to prove convergence we   resort to a compensated compactness argument, introduced by P.-L. Lions and N.  Masmoudi (see e.g. \cite{L-M}) in the study of
incompressible limit and first employed by the second author and L. Saint-Raymond in \cite{G-SR_2006} in the context of rotating fluids with variations of the rotation axis.
Let us mention that our compensated compactness argument is slightly different from the one used in e.g. \cite{F-G-GV-N} (for barotropic compressible Navier-Stokes
equations), and it allows us to treat also the possible presence of vacuum regions.

This method enables us to prove the convergence of the convective term: similarly to~\cite{G-SR_2006} and \cite{F-G-GV-N}, we deduce that the transport
term vanishes in the limit of fast rotation, up to some remainders associated with an additional constraint which has to be satisfied by the limit density and velocity field.
The explanation is that, as already remarked in the above mentioned papers, having a variable limit density imposes a strong constraint on the motion, which translates into the fact that
the kernel of the penalized operator is smaller. As a consequence, the limit equations become linear: this remarkable fact can be interpreted as a sort of turbulent behaviour
of the fluid, where all the scales are mixed and one can identify only an averaged dynamics. In our setting, this will be even more apparent: let us briefly justify this claim.
Essentially due to the very weak information on the density variations~$\sigma_\veps$, we are not able to pass to the limit directly in the original formulation
of our system: the major problems reside in the convergence of the Coriolis term and in the fact that we do not dispose of an explicit equation for the $\sigma_\veps$'s.
As a consequence, exploiting the above mentioned special relation linking $\sigma_\veps$ and $\eta_\veps$, we are only able to prove convergence on the vorticity formulation
of our system, which exactly mixes these two quantities. This is similar to what happens in the compressible case, but in our context we dispose of {no relations}
linking the velocity field $u_\veps$ and the density variations: so this equation provides information only on the special combination of the limit quantities $\sigma$ and
$\eta$, but not on the dynamics of the limit velocity field $u$ and $\sigma$ separately. See also   Remark \ref{r:non-hom_lim} below for further comments about this point.

As a last comment, let us mention that, in Subsection \ref{ss:dens-full}, we will show a convergence result to the full system, where the dynamics of the limit density variation $\sigma$ and
of the limit velocity field $u$ are decoupled. However, this statement is only a conditional convergence result, since it requires strong \textsl{a priori} bounds on the
family of velocity fields. These bounds involve higher regularity than the one we can obtain by classical energy estimates, and   is not propagated uniformly in $\veps$ in general,
due to the singular behaviour of the Coriolis term.
 
\medbreak
Let us conclude this introduction with a short overview of the paper.

In the next section we formulate our assumptions and main results. In Section \ref{s:sing-pert} we study the singular perturbation operator: namely we establish uniform bounds
for the family $\bigl(\rho_\veps,u_\veps\bigr)_\veps$ and the constraints their limit points have to satisfy. Section \ref{s:hom} is devoted to proving convergence in the slightly
non-homogeneous case, while in Section \ref{s:dens} we pass to the limit in the fully non-homogeneous framework.
We postpone to Appendix \ref{app:LP} some tools from Littlewood-Paley theory and paradifferential calculus, which are needed in our analysis.

%%%%%%%%%%%%%%%%%%%%%%%%%%%%%%%%%%%%%%%%%%
%%%%%%%%%%%%%%%%%%%%%%%%%%%%%%%%%%%%
\subsubsection*{Acknowledgements}
%%%%%%%%%%%%%%%%%%%%%%%%%%%%%%%%%%%%%
%%%%%%%%%%%%%%%%%%%%%%%%%%%%%%%%%%

The work of the first author has been partially supported by the LABEX MILYON (ANR-10-LABX-0070) of Universit\'e de Lyon, within the program ``Investissement d'Avenir''
(ANR-11-IDEX-0007), and by the project BORDS, both operated by the French National Research Agency (ANR).

%%%%%%%%%%%%%%%%%%%%%%%%%%%%%%%%%%%%%%%%%%%%%%%%%%
%%%%%%%%%%%%%%%%%%%%%%%%%%%%%%%%%%%%%%%%%%%%%%%%%%
\section{Assumptions and results} \label{s:results}
%%%%%%%%%%%%%%%%%%%%%%%%%%%%%%%%%%%%%%%%%%%%%%%%%%
%%%%%%%%%%%%%%%%%%%%%%%%%%%%%%%%%%%%%%%%%%%%%%%%%%

In the present section we introduce our working assumptions, and after discussing briefly the existence of weak solutions to our system, we formulate our main results.

\subsection{Main assumptions} \label{ss:hyp}

We state here the main assumptions on the initial data. 
For this and the next subsection (where we discuss the existence of weak solutions), we refer to Subsection 2.1 of \cite{Lions_1} for more details.

Given a small parameter $\veps\,\in\,]0,1]$, let us consider, in the domain
$$\Omega\;:=\;\R^2\quad\mbox{ or }\quad\T^2\,,$$
the non-homogeneous incompressible Navier-Stokes equations with Coriolis force
\begin{equation} \label{eq:dd-NSC}
\left\{\begin{array}{l}
        \d_t\rho\,+\,\mbox{div} \,(\rho   u)\,=\,0 \\[1ex]
        \d_t(\rho\,u)\,+\,\mbox{div}\,(\rho\,u\otimes   u)\,+\,\dfrac{1}{\veps}\,\nabla\Pi\,+\,\dfrac{1}{\veps}\,\rho\,u^\perp\,-\,
\nu\,\Delta u\,=\,0 \\[1ex]
\div u\,=\,0\,.
       \end{array}
\right.
\end{equation}
The scalar function $\rho\geq0$ represents the density of the fluid and $u$ is its velocity field; the scalar function $\Pi$ is the pressure of the fluid;
finally, the term $\rho\,u^\perp$ is due to the action of the Coriolis force on the fluid.
We recall that the term $\nabla\Pi$ can be interpreted as the Lagrangian multiplier related to the incompressibility constraint $\div u\,=\,0$.

For any fixed value of $\veps$, we consider a weak solution $\bigl(\rho_\veps,u_\veps\bigr)$ to the previous system, in the sense specified
by Definition \ref{d:weak_2} below.
For this, we supplement system \eqref{eq:dd-NSC} with general \emph{ill-prepared} initial data:   for the density functions, we take
\begin{equation} \label{hyp:rho_0}
\rho_{0,\veps}\,=\,\rho_0\,+\,\veps\,r_{0,\veps}\,,\qquad\qquad\mbox{ with }\qquad \rho_0\in\mc{C}^2_b\quad\mbox{ and }\quad
\bigl(r_{0,\veps}\bigr)_\veps\,\subset\,L^2(\Omega)\cap L^\infty(\Omega)\,.
\end{equation}
 Here and in the following, the notation $\bigl(a_{\veps}\bigr)_\veps\,\subset\,X$ (for some Banach space~$X$) is to be understood as the fact that the
sequence~$\bigl(a_{\veps}\bigr)_\veps$ is uniformly bounded in~$X$.  The symbol~$\mc{C}^2_b$ denotes the subspace of~$\mc C^2$ functions which are bounded as well  as their first and second order derivatives; this condition on $\rho_0$ can be somehow relaxed in
the spirit of \cite{F_JMFM}, but we assume it for simplicity.

Furthermore,  we assume also that there exists a positive constant~$\rho^*>0$ such that
$$
0\,\leq\,\rho_0\,\leq\,\rho^*\qquad\qquad\mbox{ and }\qquad\qquad 0\,\leq\,\rho_{0,\veps}\,\leq\,2\,\rho^*\,,
$$
where $\rho^*$ is a positive real number. The previous assumption means in particular that we are considering the possible presence of vacuum.
Notice that the positivity requirement on $\rho_0$ is a sort of compatibility condition, due to the fact that $\rho_{0,\veps}\,\ra\,\rho_0$ when $\veps$ goes to $0$.
Of course, this imposes some conditions also on the sign of the perturbation functions $r_{0,\veps}$ when $\rho_0$ vanishes, but this is not important at this stage.

\begin{rem} \label{r:dens_0}
Observe that, in the case when $\rho_0\geq\rho_*>0$ (and especially when $\rho_0\equiv\rho_*$), with no loss of generality we can suppose that
$$ %\begin{equation} \label{est:dens_0}
0\,<\,\rho_*/2\,\leq\,\rho_{0,\veps}\,\leq\,2\,\rho^*\qquad\qquad\mbox{for all }\quad \veps\,\in\,]0,1]\,.
$$ %\end{equation}
\end{rem}

For all $\veps$, we also take an initial momentum $m_{0,\veps}$ such that $m_{0,\veps}=0$ almost everywhere on the set $\bigl\{\rho_{0,\veps}=0\bigr\}$, and  
$$
\bigl(m_{0,\veps}\bigr)_\veps\,\subset\,L^2(\Omega)\qquad\qquad\mbox{ and }\qquad\qquad \left(\left|m_{0,\veps}\right|^2/\rho_{0,\veps}\right)_\veps\,\subset\,L^1(\Omega)\,,
$$
where we agree that $\left|m_{0,\veps}\right|^2/\rho_{0,\veps}=0$ on the set $\bigl\{\rho_{0,\veps}=0\bigr\}$.

\begin{rem} \label{r:vel_0}
In general, one would like to say that $m_{0,\veps}\,=\,\rho_{0,\veps}\,u_{0,\veps}$. Nonetheless, conditions are imposed on $m_{0,\veps}$ rather than on the velocity fields
$u_{0,\veps}$, since these fields may be not defined when the densities $\rho_{0,\veps}$ vanish.
For the same reason, we do not impose any divergence-free condition on the initial velocities. We refer to the discussion of \cite{Lions_1} (see page 25 therein) for more details
about this issue.

Of course, in the absence of vacuum (keep in mind Remark \ref{r:dens_0} above), the previous assumptions on $m_{0,\veps}$ can be reformulated in an easier way, as conditions
on the ``true'' velocity fields: we suppose then
$$
\bigl(u_{0,\veps}\bigr)_\veps\,\subset\,L^2(\Omega)\,,\qquad\qquad\mbox{ with }\qquad \div u_{0,\veps}\,=\,0\,.
$$
\end{rem}

Finally, in the case when $\Omega\,=\,\R^2$, we additionally require that there exists  $\delta>0$ such that, for all $\veps\in\,]0,1]$, one has the uniform embedding
\begin{equation} \label{hyp:vacuum_1}
\left(\frac{1}{\rho_{0,\veps}}\;\mathds{1}_{\{\rho_{0,\veps}<\delta\}}\right)_\veps\;\subset\;L^1(\Omega)\,,
\end{equation}
where $\mathds{1}_A$ denotes the characteristic function of a set $A\subset\R^2$. Alternatively, we suppose that there exist a $\oline{\rho}>0$ and a $p_0\in\,]1,+\infty[\,$ such that
\begin{equation} \label{hyp:vacuum_2}
\left(\bigl(\oline{\rho}\,-\,\rho_{0,\veps}\bigr)^+\right)_\veps\;\subset\;L^{p_0}(\Omega)\,,
\end{equation}
where we have denoted by $f^+$ the positive part of a function $f$.
We point out that other assumptions can be considered on the initial densities: we refer once again to Chapter 2 of \cite{Lions_1}.

Up to passing to subsequences, we can assume that
\begin{equation} \label{eq:conv-initial}
r_{0,\veps}\,\stackrel{*}{\rightharpoonup}\,r_0\quad\mbox{ in }\;L^2(\Omega)\cap L^\infty(\Omega)\qquad\qquad\mbox{ and }\qquad\qquad
m_{0,\veps}\,\rightharpoonup\,m_0\quad\mbox{ in }\;L^2(\Omega)\,.
\end{equation}
Of course, in absence of vacuum, we can suppose that
\begin{equation} \label{eq:conv-u}
u_{0,\veps}\,\rightharpoonup\,u_0\quad\mbox{ in }\;L^2(\Omega)\,.
\end{equation}

Before stating our main results, let discuss briefly energy estimates, and the existence of weak solutions for our equations.

\subsection{Energy inequality, finite energy weak solutions} \label{ss:weak}

Suppose for a while that $(\rho,u)$ are smooth solutions to system \eqref{eq:dd-NSC}, related to smooth initial data~$(\rho_0,m_0)$.
First of all, we notice that the density is simply transported by a divergence-free velocity field: so all its $L^q$ norms are preserved in time.
On the other hand, system \eqref{eq:dd-NSC} has a conserved energy:  for all $t>0$
\begin{equation} \label{est:energy}
\left\|\sqrt{\rho(t)}\,u(t)\right\|^2_{L^2(\Omega)}\,+\,2\,\nu\int^t_0\left\|\nabla u(\tau)\right\|_{L^2(\Omega)}^2\,d\tau\,\leq\,\left\||m_0|^2/\rho_0\right\|_{L^1}\,.
\end{equation}
Notice that, for smooth enough solutions, inequality \eqref{est:energy} becomes actually an equality.
We refer to Subsection \ref{ss:bounds} below for more details. However, the previous considerations motivate the following definition of \emph{weak solution} to our system.

\begin{defin} \label{d:weak_2}
Fix initial data $(\rho_0,m_0)$ such that the conditions listed in Subsection \ref{ss:hyp} above are fulfilled.

We say that $\bigl(\rho,u\bigr)$ is a \emph{weak solution} to system \eqref{eq:dd-NSC}
in $[0,T[\,\times\Omega$ (for some time $T>0$) with initial datum $(\rho_0,m_0)$ if the following conditions are satisfied:
\begin{itemize}
 \item[(i)] $\rho\,\in\,L^{\infty}\bigl([0,T[\,\times\Omega\bigr)$ and $\rho\,\in\,\mc C\bigl([0,T[\,;L^q_{\rm loc}(\Omega)\bigr)$ for  all $1\leq q<+\infty$;
\item[(ii)] $\rho\,|u|^2\,\in L^\infty\bigl([0,T[\,;L^1(\Omega)\bigr)$, with $\nabla u\,\in L^2\bigl([0,T[\,\times\Omega\bigr)$ and $u\,\in\,L^2\bigl([0,T[\,\times\Omega\bigr)$;
\item[(iii)] the mass equation is satisfied in a weak sense: for any $\phi\in\mc{D}\bigl([0,T[\,\times\Omega\bigr)$ one has
\begin{equation} \label{eq:weak-mass}
-\int^T_0\int_{\Omega}\biggl(\rho\,\d_t\phi\,+\,\rho\,u\,\cdot\,\nabla\phi\biggr)\,dx\,dt\,=\,\int_{\Omega}\rho_0\,\phi(0)\,dx\,;
\end{equation}
\item[(iv)] the divergence-free condition on $u$ is satisfied in $\mc{D}'\bigl(\,]0,T[\,\times\Omega\bigr)$;
\item[(v)] the momentum equation is satisfied in a weak sense: for any $\psi\in\mc{D}\bigl([0,T[\times\Omega\bigr)$ such that~$\div\psi=0$, one has
\begin{equation}
\int^T_0\int_{\Omega}\biggl(-\rho\,u\cdot\d_t\psi\,-\,\rho\,u\otimes u:\nabla\psi\,+\,\frac{1}{\veps}\,\rho\,u^\perp\cdot\psi\,+\,
\nu\,\nabla u:\nabla\psi\biggr)\,dx\,dt\,=\,\int_{\Omega}m_0\cdot\psi(0)\,dx\,; \label{eq:weak-momentum}
\end{equation}
\item[(vi)] for almost every $t\in\,]0,T[\,$, the energy inequality \eqref{est:energy} holds true.
\end{itemize}
The solution is \emph{global} if the previous conditions are satisfied for all $T>0$.
\end{defin}

Notice that point (vi) of the previous definition prescribes a finite energy condition on $(\rho,u)$, which are then solutions
\textsl{\`a la Leray} (see \cite{Leray}). Such a condition is the basis for the investigation of the singular perturbation problem, since   
 most of the \textsl{a priori} bounds on our family of weak solutions will be derived from it.

The next result guarantees the existence of a global in time weak solution to the non-homogeneous Navier-Stokes equations with Coriolis force,
for any fixed value of the parameter~$\veps\in\,]0,1]$.
\begin{thm} \label{t:exist}
Fix $\veps\in\,]0,1]$ and consider an initial datum $(\rho_0,m_0)$ as in Definition~{\rm\ref{d:weak_2}} above.
Then there exists a global in time weak solution $(\rho,u)$ to equations~\eqref{eq:dd-NSC}.
\end{thm}

For a proof of the previous statement, we refer to Chapter 2 of \cite{Lions_1}. We remark that the presence of the Coriolis term can easily  be  handled: indeed, it  vanishes identically vanishes
in the energy estimates (since it is orthogonal to $u$), and passing to the limit in it in the compactness argument asks for a similar effort as in treating the time derivative, for instance.

\subsection{Statement of the results} \label{ss:results}

We now consider a family of initial data $\bigl(\rho_{0,\veps},m_{0,\veps}\bigr)_{\veps\in\,]0,1]}$ satisfying all the assumptions stated in Subsection \ref{ss:hyp} above.
For any fixed~$\veps\in\,]0,1]$, Theorem \ref{t:exist} provides  a global in time weak solution~$(\rho_\veps,u_\veps)$ to system \eqref{eq:dd-NSC}.
Our main purpose is to characterize the dynamics in the limit for~$\veps$ going to $0$, namely the system of equations verified by the limit points of the family
$\bigl(\rho_\veps,u_\veps\bigr)_\veps.$

It turns out that two different scenarios may occur:
\begin{itemize}
 \item[(i)] the initial density is a small perturbation of a constant state, namely $\rho_0\equiv1$ (or another positive constant);
\item[(ii)] the reference state of density is truly non-constant, meaning (roughly) that $\nabla\rho_0\neq0$: the precise assumption which is needed is given in condition
\eqref{eq:non-crit} below.
\end{itemize}

In the former case, since the constant density configuration is transported by the velocity field, we are led to studying the incompressible and homogeneous limits simultaneously.
Although the non-homogeneity $\rho_{0,\veps}$ requires  some effort in order to be handled, a not so difficult adaptation of the arguments used in \cite{G-SR_2006}
allows us to prove convergence to $2$-D homogeneous incompressible Navier-Stokes equations.
More precisely, we have the following statement, which will be proved in Section \ref{s:hom}. Remark that we can formulate assumptions directly on the initial velocities $u_{0,\veps}$
(recall Remark \ref{r:vel_0} above), since   the density does not vanish.

\begin{thm} \label{t:hom}
Let us set $\rho_0\,=\,1$ and consider a family of initial data $\bigl(\rho_{0,\veps},u_{0,\veps}\bigr)_\veps$ satisfying the assumptions fixed in Subsection~{\rm\ref{ss:hyp}}.
Let $\bigl(\rho_\veps\,,\,u_\veps\bigr)_\veps$ be a family of corresponding weak solutions to system
\eqref{eq:dd-NSC} in $\R_+\times\Omega$, as given by Theorem~{\rm\ref{t:exist}}.
Let us define $u_0$ and $r_0$ as done in \eqref{eq:conv-initial}-\eqref{eq:conv-u}, and set~$r_\veps\,:=\,\veps^{-1}\left(\rho_\veps-1\right)$.

Then, there exist $r\in L^\infty\bigl(\R_+;L^2(\Omega)\cap L^\infty(\Omega)\bigr)$ and $u\,\in\,L^\infty\bigl(\R_+;L^2(\Omega)\bigr)\,\cap\,L^2_{\rm loc}\bigl(\R_+;H^1(\Omega)\bigr)$,
with $\div u=0$ in the distributional sense, such that, up to the extraction of a subsequence, for any time $T>0$ fixed, one has the following convergence properties:
\begin{itemize}
 \item[(a)] $r_\veps\,\stackrel{*}{\rightharpoonup}\,r$ in $L^\infty\bigl([0,T];L^2(\Omega)\cap L^\infty(\Omega)\bigr)$;
 \item[(b)] $u_\veps\,\stackrel{*}{\rightharpoonup}\,u$ in $L^\infty\bigl([0,T];L^2(\Omega)\bigr)\,\cap\,L^2\bigl([0,T];H^1(\Omega)\bigr)$.
\end{itemize}
Moreover, $u$ is a global in time weak solution of the homogeneous Navier-Stokes equations with initial datum $u_0$, while $r$ satisfies (still in the weak sense) a pure transport equation by $u$,
with initial datum $r_0$. Namely, the following equations hold true in the weak sense:
\begin{equation} \label{eq:hom_lim}
\left\{\begin{array}{l}
        \d_tr\,+\,{\mbox{\rm div} }\,(r \,  u)\,=\,0 \\[1ex]
        \d_tu\,+\,\div(u\otimes u)\,+\,\nabla\Pi\,+\,r\,u^\perp\,-\,\nu\,\Delta u\,=\,0 \\[1ex]
         \div u\,=\,0\,,
       \end{array}
\right.
\end{equation}
with initial data $r_{|t=0}\,=\,r_0$ and $u_{|t=0}\,=\,u_0$, for a suitable pressure function $\Pi$. \\
If, in addition, $r_0\in H^{1+\beta}(\Omega)$, for some $\beta>0$, and $u_0\in H^1(\Omega)$, then the solution $(r,u,\Pi)$ to system \eqref{eq:hom_lim}, related
to the initial datum $(r_0,u_0)$, is unique. Hence, the convergence holds for the whole sequence $\bigl(r_\veps,u_\veps\bigr)_\veps$.
\end{thm}

The previous result is not really surprising, since it allows us to recover for slightly non-homogeneous fluids (almost) the same result as for homogeneous ones (see also Remark
\ref{r:1_vort} below). However, its proof is interesting since it will suggest how to argue in the truly non-constant density case, where we work in a very low regularity framework.

Let us comment also on the uniqueness for the limit system. System \eqref{eq:hom_lim} is studied in Subsection \ref{ss:hom_limit} below: Theorem \ref{t:hom_r-u}  
specifies the functional class where the solution is unique. We   also show energy and stability estimates for those equations.

\begin{rem} \label{r:1_vort}
Compared to the homogeneous case, in the limit equation \eqref{eq:hom_lim} we notice the presence of the additional term $r\,u^\perp$,
which is essentially due to variations of the density, as explained in the Introduction.
Notice that one can pass to the vorticity formulation of the previous equations: denoting the vorticity of~$u$ by~$\omega:=\curl u :=\d_1u^2-\d_2u^1$ and using also the first relation
in \eqref{eq:hom_lim}, we find
$$
\d_t\bigl(\omega\,-\,r\bigr)\,+\,u\cdot\nabla\omega\,-\,\nu\,\Delta\omega\,=\,0\,.
$$
Using the Biot-Savart law $u\,=\,-\,(-\Delta)^{-1}\nabla^\perp\omega$, we can remark the strong analogy of the previous equation with the one usually found
for compressible fluids, see e.g.  \cite{F-G-GV-N}, \cite{F-G-N} (see also \cite{F_MA} for a modified formulation, due to capillary effects in the limit).
There, the vorticity formulation is more convenient in passing to the limit, since one disposes of the additional relation $u=\nabla^\perp r$
(namely $r$ is a stream function for $u$), which gives $\omega=\Delta r$.
\end{rem}

\medbreak
Let us now turn our attention to the case of an effectively variable reference density. Similarly to what happens for compressible fluids (see e.g. \cite{F-G-GV-N}),
having a non-constant density in the limit dynamics imposes a strong constraint on weak-limit points of the family of solutions: the consequence of this fact is that the
convective term vanishes in the limit and the final equations become linear. Analogous effects can be noticed in
the case of variations of the rotation axis (see e.g. \cite{G-SR_2006} and \cite{F_JMFM}).
% This is a remarkable fact, which can be seen (as already remarked in \cite{G-SR_2006}) as an expression of a sort of turbulent behaviour, where all the scales are
%mixed and one can identify only an averaged dynamics.

For technical reasons, we need to assume that the reference density $\rho_0$ has non-degenerate critical points, in the sense of \cite{G-SR_2006}.
Namely, we suppose
\begin{equation} \label{eq:non-crit}
\lim_{\delta\ra0}\;\mu\left(\left\{x\,\in\,\Omega\;\Bigl|\;\bigl|\nabla\rho_0(x)\bigr|\,\leq\,\delta\right\}\right)\,=\,0\,,
\end{equation}
where  $\mu(\mc O)$ denotes the Lebesgue measure of a set $\mc O\,\subset\R^2$.

\begin{thm} \label{t:dens}
Assume the reference density $\rho_0$ satisfies condition \eqref{eq:non-crit}, and consider a family of initial data $\bigl(\rho_{0,\veps},u_{0,\veps}\bigr)_\veps$ satisfying
the assumptions fixed in Subsection~{\rm\ref{ss:hyp}}. Let $\bigl(\rho_\veps\,,\,u_\veps\bigr)_\veps$ be a family of corresponding weak solutions to system
\eqref{eq:dd-NSC} in $\R_+\times\Omega$, as given by Theorem~{\rm\ref{t:exist}}.
Let us define~$m_0$ and $r_0$ as  in \eqref{eq:conv-initial}, and set $\s_\veps\,:=\,\veps^{-1}\left(\rho_\veps-\rho_0\right)$.

Then, there exist $\s\in L^\infty\bigl(\R_+;H^{-2-\delta}(\Omega)\bigr)$, for $\delta>0$ arbitrarily small, and a vector field~$u$ in the
space~$L^2_{\rm loc}\bigl(\R_+;H^1(\Omega)\bigr)$,
with $\div u=\div\bigl(\rho_0\,u\bigr)=0$, such that, up to the extraction of a subsequence, for any  fixed time $T>0$, one has the following convergence properties:
\begin{itemize}
 \item[(a)] $\rho_\veps\,\stackrel{*}{\rightharpoonup}\, {\rho}_0$ in $L^\infty\bigl([0,T];L^2(\Omega)\cap L^\infty(\Omega)\bigr)$;
%, and the convergence is strong in $L^\infty\bigl([0,T];H^{-\delta}_{\rm loc}(\Omega)\bigr)$ for all indexes $\delta>0$ {\tt ce dernier point n'apporte pas grand chose \`a l'\'enonc\'e};
 \item[(b)] $u_\veps\,\rightharpoonup\,u$ in $L^2\bigl([0,T];H^1(\Omega)\bigr)$;
 \item[(c)] $\s_\veps\,\stackrel{*}{\rightharpoonup}\,\s$ in $L^\infty\bigl([0,T];H^{-2-\delta}(\Omega)\bigr)$ for all $\delta>0$.
\end{itemize}
Moreover, set $\omega\,:=\,\curl u$ and $\eta\,:=\,\curl\bigl(\rho_0\,u\bigr)$: then there exists a distribution $\Gamma\,\in\,\mc D'\bigl(\R_+\times\Omega\bigr)$ such that
the equation %, linking $\sigma$ and $u$ \textsl{via} $\omega$ and $\eta$,
\begin{equation} \label{eq:dens_lim}
\d_t (\eta\,-\,\s )\,-\,\nu\,\Delta\omega\,+\,\curl\bigl(\rho_0\,\nabla\Gamma\bigr)\,=\,0\,,
\end{equation}
supplemented with the initial condition $(\eta\,-\,\s)_{|t=0}\,=\,\curl m_0\,-\,r_0$, is verified in the weak sense.
%where~$u_0$ and $r_0$ have been defined in \eqref{eq:conv-initial}.
\end{thm}

%Let us make some remarks about Theorem \ref{t:dens}, which is proved in Section \ref{s:dens}.

\begin{rem} \label{r:wek-conv}
Let us stress the fact that we are able to prove a true weak convergence of system~\eqref{eq:dd-NSC} to equation \eqref{eq:dens_lim}. More precisely,
the convergence holds true whenever  the mass and momentum equations in \eqref{eq:dd-NSC} are tested respectively on scalar test functions $\phi$ and on vector-valued divergence free test functions
$\psi\,=\,\nabla^\perp\phi$  (without requiring any other constraint). 
%Then, the class of test functions for which we can prove convergence is the same which is admissible for
%the primitive system, with no further restrictions.
\end{rem}

\begin{rem} \label{r:eta-omega}
Let us point out that equation \eqref{eq:dens_lim} can be formulated in terms of $\omega$ and $\sigma$ only, using the following relation:
$$
\eta\,=\,-\,\div\left(\rho_0\,\nabla(-\Delta)^{-1}\omega\right)\,,
$$
which easily follows from writing $\rho_0\,u\,=\,\nabla^\perp\lam_1$ and $u\,=\,\nabla^\perp\lam_2$ (which both hold true thanks to the divergence-free conditions on the limit). Indeed,
on the one hand, taking the $\curl$ of both equations yields $\eta\,=\,\Delta\lam_1$ and $\omega\,=\,\Delta\lam_2$; on the other hand, the relation $\rho_0\,\nabla\lam_2\,=\,\nabla\lam_1$
implies that $\div\bigl(\rho_0\,\nabla\lam_2\bigr)\,=\,-\div\bigl(\rho_0\,\nabla(-\Delta)^{-1}\omega\bigr)\,=\,\Delta\lam_1\,=\,\eta$.
\end{rem}

\begin{rem} \label{r:non-hom_lim}
In terms of the limit density fluctuation $\sigma$ and the limit velocity field $u$, the final equations~(\ref{eq:dens_lim}) can be written formally as
\begin{equation} \label{eq:dens_lim-full}
\left\{\begin{array}{l}
        \d_t\s\,+\,u\cdot\nabla\s\,=\,0 \\[1ex]
        \rho_0\,\d_tu\,+\,\nabla\Pi\,+\,\rho_0\,\nabla\Gamma\,+\,\s\,u^\perp\,-\,\nu\,\Delta u\,=\,0 \\[1ex]
         \div u\,=\,\div\bigl(\rho_0\,u\bigr)\,=\,0\,.   %%%%%\,,
       \end{array}
\right. 
\end{equation}
%%%%% and they have to be satisfied in the weak sense.
%%%%% In particular, the momentum equation holds true when tested against $\psi\,\in\,\mc{D}\bigl([0,T[\,\times\Omega\bigr)$ such that $\div\psi\,=\,0$.
We remark that the term $\rho_0\,\nabla\Gamma$ in the second relation can be understood as the Lagrangian multiplier associated with the constraint $\div\bigl(\rho_0\,u\bigr)\,=\,0$.
Roughly speaking, this term comes from the Coriolis operator when considered at the highest order, in the limit of fast rotation: see also Remark \ref{r:constr-cor} about this point.

Nonetheless, contrary to the case of a constant reference density, here we do not have enough regularity on the density oscillations $(\s_\veps)_\veps$ to prove the convergence
on the velocity equations and to find an equation   for the $\sigma_\veps$'s themselves. So the derivation of the previous system is just formal and, in order to rigorously pass to the limit,
we are forced to resort to the vorticity formulation.

We postpone to Subsection \ref{ss:dens-full} (see Theorem~\ref{t:dens-full} therein) the statement of a conditional result where, under suitable assumptions, we are able to show convergence
to (a modified version of) the full system \eqref{eq:dens_lim-full}.
\end{rem}

%%%%%%%%%%%%%%%%%%%%%%%%%%%%%%%%%%%%%%%%%%%%%%%%%%
%%%%%%%%%%%%%%%%%%%%%%%%%%%%%%%%%%%%%%%%%%%%%%%%%%
\section{Study of the singular perturbation} \label{s:sing-pert}
%%%%%%%%%%%%%%%%%%%%%%%%%%%%%%%%%%%%%%%%%%%%%%%%%%
%%%%%%%%%%%%%%%%%%%%%%%%%%%%%%%%%%%%%%%%%%%%%%%%%%

In the present section we establish important properties on the family of weak solutions we are considering. First we derive uniform bounds, and 
then we look for conditions their limit points have to satisfy; finally, we return to $\bigl(\rho_\veps,u_\veps\bigr)_\veps$ and we infer additional properties, especially in the case of
non-constant $\rho_0$.

\subsection{Uniform bounds} \label{ss:bounds}

We establish here first uniform bounds on the family of weak solutions $\bigl(\rho_\veps,u_\veps\bigr)_\veps$ to system \eqref{eq:dd-NSC}.
The formal arguments we   use are fully justified for smooth solutions, and the  bounds obtained pass to weak solutions thanks to a standard approximation procedure
(see for instance Chapter~2 of~\cite{Lions_1}).

\subsubsection{Bounds on the density functions} \label{sss:b-dens}

To begin with, let us focus on the density functions $\rho_\veps$. By Theorem 2.1 of \cite{Lions_1}, their $L^\infty$ bounds are preserved since they satisfy a pure transport equation
by a divergence-free vector field. Therefore we get, for any $\veps\in\,]0,1]$ and for almost every $(t,x)\in\R_+\times\Omega$,
\begin{equation} \label{est:dens_inf}
0\,\leq\,\rho_{\veps}(t,x)\,\leq\,2\,\rho^*\,.
\end{equation}
In particular, $(\rho_\veps)_\veps$ is uniformly bounded in the space $L^\infty\bigl(\R_+\times\Omega\bigr)$, and hence one gathers the existence of a positive function
$\rho\,\in\,L^\infty\bigl(\R_+\times\Omega\bigr)$, which satisfies the same bounds as \eqref{est:dens_inf}, and such that (up to the extraction of a subsequence)
\begin{equation} \label{cv:weak-rho}
\rho_\veps\;\stackrel{*}{\rightharpoonup}\;\rho \quad\mbox{ in }\quad L^\infty\bigl(\R_+;L^1_{\rm loc}(\Omega)\cap L^\infty(\Omega)\bigr)\,.
\end{equation}

Let us focus for a while on the slightly non-homogeneous case, i.e. when $\rho_0\,\equiv\,1$. Defining the quantity $r_\veps\,:=\,\veps^{-1}\,(\rho_\veps-1)$ as
in the statement of Theorem \ref{t:hom} above, the mass equation can be rewritten as
\begin{equation} \label{eq:r_e}
\d_tr_\veps\,+\,\,\mbox{div} \,(  r_\veps\, u_\veps) \,=\,0\;,\qquad\qquad (r_\veps)_{|t=0}\,=\,r_{0,\veps}\,.
\end{equation}
From this equation and the assumption on the initial data~\eqref{hyp:rho_0}, we deduce the uniform bounds
\begin{equation} \label{est:1_dens-L^p}
\bigl(r_\veps\bigr)_\veps\,\subset\,L^\infty\bigl(\R_+;L^2(\Omega)\cap L^\infty(\Omega)\bigr)\,,
\end{equation}
and therefore there exists a $r\,\in\,L^\infty\bigl(\R_+;L^2(\Omega)\cap L^\infty(\Omega)\bigr)$ such that, up to an extraction, $r_\veps\,\stackrel{*}{\rightharpoonup}\,r$
in this space.

Let us now consider the general non-homogeneous case. Suppose that $u_\veps$ is smooth. Then we can introduce its flow $\psi_\veps(t,x)\,\equiv\,\psi_{\veps,t}(x)$, defined by the formula
$$
\psi_{\veps,t}(x)\,:=\,x\,+\,\int^t_0u_\veps\bigl(\tau,\psi_{\veps,\tau}(x)\bigr)\,d\tau
$$
for all $(t,x)\in\R_+\times\Omega$. Then, by the transport equation for the density, we deduce that
$$
\rho_\veps(t,x)\,=\,\rho_{0,\veps}\left(\psi_{\veps,t}^{-1}(x)\right)\,=\,\wtilde{\rho}_\veps(t,x)\,+\,\veps\,r_\veps(t,x) \, , 
$$
where, in analogy to what was done above, we have defined
$$
\wtilde{\rho}_\veps(t,x)\,=\,\rho_{0}\left(\psi_{\veps,t}^{-1}(x)\right)\qquad\qquad\mbox{ and }\qquad\qquad
r_\veps(t,x)\,=\,r_{0,\veps}\left(\psi_{\veps,t}^{-1}(x)\right)\,.
$$
Namely both $\rho_0$ and $r_{0,\veps}$ are transported by $u_\veps$; correspondingly, $\wtilde{\rho}_\veps$ and $r_\veps$ verify a pure transport equation
$\d_ta\,+\,u_\veps\cdot\nabla a\,=\,0$, from which we deduce
$$
0\,\leq\,\wtilde{\rho}_\veps\,\leq\,\rho^*\qquad\qquad\mbox{ and }\qquad\qquad \bigl(r_\veps\bigr)_\veps\,\subset\,L^\infty\bigl(\R_+;L^2(\Omega)\cap L^\infty(\Omega)\bigr)\,.
$$
The previous uniform bounds are inherited by weak solutions, even if $u_\veps$ is no longer smooth.
However, in light of Proposition \ref{p:constr} below, it turns out that,
apart from the case $\rho_0\equiv1$ (and thus~$\wtilde{\rho}_\veps(t)\equiv1$ for any time $t\geq0$), the decomposition $\rho_\veps\,=\,\wtilde{\rho}_\veps+\veps r_\veps$
is not suitable for the analysis of the singular perturbation problem. We refer to Subsection \ref{ss:further} for more comments about this point.

\subsubsection{Estimates for the velocity fields, and consequences} \label{sss:b-vel}

We turn now our attention to the momentum equation. Thanks to~(\ref{est:energy}) we have the bounds
$$
\left\|\sqrt{\rho_\veps(t)}\,u_\veps(t)\right\|^2_{L^2(\Omega)}\,+\,2\,\nu\int^t_0\left\|\nabla u_\veps(\tau)\right\|_{L^2(\Omega)}^2\,d\tau\,\leq\,
C\,\left\||m_{0,\veps}|^2/\rho_{0,\veps}\right\|_{L^1} %\left\|\sqrt{\rho_{0,\veps}}\,u_{0,\veps}\right\|^2_{L^2(\Omega)}
$$
for all $t>0$. Notice that, by assumption on the initial data, the right-hand side of the previous inequality is uniformly bounded. 
Hence, keeping in mind property \eqref{est:dens_inf} on the density, we easily deduce that
\begin{equation} \label{ub:u_L^p}
\bigl(\sqrt{\rho_\veps}\,u_\veps\bigr)_\veps\;\subset\;L^\infty\bigl(\R_+;L^2(\Omega)\bigr)\qquad\mbox{ and }\qquad
\bigl(\nabla u_\veps\bigr)_\veps\;\subset\;L^2\bigl(\R_+;L^2(\Omega)\bigr)\,,
\end{equation}
uniformly in $\veps\in\,]0,1]$.

Let us focus on the case $\Omega=\R^2$ for a while. Conditions \eqref{hyp:vacuum_1} and \eqref{hyp:vacuum_2} are preserved for $\rho_\veps$, as a consequence of the pure transport equation by a
divergence-free velocity field. So, referring to point 8 in Remark 2.1 of \cite{Lions_1}, we also infer the inclusion
\begin{equation} \label{ub:u_L^2}
\bigl(u_\veps\bigr)_\veps\,\subset\,L^2\bigl([0,T[\,\times\Omega\bigr)
\end{equation}
for any fixed time $T>0$. %, uniformly for $\veps\in\,]0,1]$.
For notational convenience, in the following we denote, for any~$p$ and any Banach space~$X$,
$$
L^p_T(X)\,:=\,L^p\bigl([0,T[\,;X(\Omega)\bigr) \,.
$$
Then we deduce that $\bigl(u_\veps\bigr)_\veps$ is a bounded family in~$L^2_T(H^1)$ for any $T>0$.
On the other hand, in the case $\Omega=\T^2$ this property follows by combining \eqref{ub:u_L^p}, Sobolev embeddings and Gagliardo-Nirenberg inequality (see Proposition \ref{p:Gagl-Nir}
in the Appendix).

Therefore, there exists~$u\,\in\,L^2_{\rm loc}\bigl(\R_+;H^1(\Omega)\bigr)$ such that,
up to extraction of a subsequence, there holds~$u_\veps\,\rightharpoonup\,u$ in this space, when $\veps\ra0$.

\medbreak
Thanks to the previous properties, we can also establish strong convergence of the densities.
More precisely, we write the mass equation under the form $ \d_t\rho_\veps\,=\,-\,\div\bigl(\rho_\veps\,u_\veps\bigr)$: this implies that, for all fixed $T>0$,
$\bigl(\d_t\rho_\veps\bigr)_\veps$ is uniformly bounded in $L^\infty\bigl([0,T];H^{-1}(\Omega)\bigr)$, and
so~$\bigl(\rho_\veps\bigr)_\veps$ is   bounded in $W^{1,\infty}\bigl([0,T];H^{-1}_{\rm loc}(\Omega)\bigr)$ uniformly in $\veps$ (the local condition in $H^{-1}_{\rm loc}(\Omega)$ comes from the fact that
$\rho_{0,\veps}$ is just $L^2_{\rm loc}(\Omega)$). Then, by the Ascoli-Arzel\`a Theorem we gather
that the family~$(\rho_\veps)_\veps$ is compact in $L^\infty\bigl([0,T];H^{-1}_{\rm loc}(\Omega)\bigr)$, and hence, keeping in mind \eqref{cv:weak-rho}, by interpolation
we deduce the strong convergence property (up to passing to a suitable subsequence)
\begin{equation} \label{cv:strong-rho}
\rho_\veps\,\longrightarrow\,\rho\qquad\qquad\mbox{ in }\qquad \mc{C}^{0,1-\eta}\bigl([0,T];H^{-1+\eta}_{\rm loc}(\Omega)\bigr)
\end{equation}
for all $0\leq\eta<1$, where we have denoted by $\mc C^{0,\g}$ the space of $\g$-H\"older continuous functions, with~$\mc C^{0,\g}$ replaced by $W^{1,\infty}$ when $\g=1$.

\begin{rem} \label{r:r_strong}
We notice that a feature analogous to \eqref{cv:strong-rho} can be deduced also on the family $(r_\veps)_\veps$, in the slightly non-homogeneous case.
\end{rem}

Finally, let us establish a very simple convergence property. Since it will be repeatedly used in what follows, we choose to devote a statement to it.
\begin{lemma} \label{l:rho-u}
Let $\bigl(\rho_\veps,u_\veps\bigr)_\veps$ be a family of weak solutions to system \eqref{eq:dd-NSC}, associated  with initial data~$\bigl(\rho_{0,\veps},u_{0,\veps}\bigr)$ satisfying
the assumptions fixed in Subsection~{\rm\ref{ss:hyp}}. Let $(\rho,u)$ be a limit point of the sequence $\bigl(\rho_\veps,u_\veps\bigr)_\veps$, as identified in Subsection~{\rm\ref{ss:bounds}}.

Then the product $\bigl(\rho_\veps\,u_\veps\bigr)_\veps$ converges to $\rho\,u$ in the distributional sense, and more precisely in the weak
topology of $L^2\bigl([0,T];H^{-\kappa}_{\rm loc}(\Omega)\bigr)$, for all $0<\kappa<1$.
\end{lemma}
\begin{proof}
We have proved in \eqref{cv:strong-rho} that the family $(\rho_\veps)_\veps$ is compact in e.g. $\mc{C}^{0,\eta}\bigl([0,T];H^{-\eta}_{\rm loc}(\Omega)\bigr)$ for some~$0<\eta<1$,
and it strongly converges to $\rho$ in the previous space. On the other hand,~$(u_\veps)_\veps$ is weakly convergent in $L^2_T(H^1)$ to $u$.
Therefore, by Corollary \ref{c:product} (ii) we get $\rho_\veps\,u_\veps\,\rightharpoonup\,\rho\,u$ in  the space~$L^2_T(H^{-\eta-\delta})$ and the result follows.
\end{proof}

\subsection{Constraints on the limit} \label{ss:constraints}

In the previous part we have proved uniform bounds on the family of weak solutions $\bigl(\rho_\veps,u_\veps\bigr)_\veps$, which allow us to identify (up to extraction)
weak limits $(\rho,u)$. In the present subsection, we collect some properties these limit points have to satisfy. 
We point out that these conditions do not fully characterize   the limit dynamics. 
\begin{prop} \label{p:constr}
Let $\bigl(\rho_\veps,u_\veps\bigr)_\veps$ be a family of weak solutions to system \eqref{eq:dd-NSC}, associated with  initial data $\bigl(\rho_{0,\veps},u_{0,\veps}\bigr)$ satisfying
the assumptions fixed in Subsection~{\rm\ref{ss:hyp}}. Let $(\rho,u)$ be a limit point of the sequence $\bigl(\rho_\veps,u_\veps\bigr)_\veps$, as identified in Subsection~{\rm\ref{ss:bounds}}.

Then one deduces the relations $\rho(t)\,\equiv\,\rho_0$ for all $t\geq0$ and $\;\div u\,=\,\div\bigl(\rho_0\,u\bigr)\,=\,0$ in the sense of~$\mc{D}'\bigl(\R_+\times\Omega\bigr)$.
In particular, we have that~$u(t)\cdot\nabla\rho_0\,=\,0$ for almost every $t\geq0$.
\end{prop}

\begin{proof}
For $T>0$ fixed, let us consider a smooth function $\psi\in\mc{D}\bigl([0,T[\,\times\Omega\bigr)$ such that $\div\psi=0$, and let us test the momentum equation on $\veps\psi$. We get
\begin{eqnarray*}
& & \hspace{-1cm}
\int^T_0\int_{\Omega}\biggl(-\veps\,\rho_\veps\,u_\veps\cdot\d_t\psi\,-\,\veps\,\rho_\veps\,u_\veps\otimes u_\veps:\nabla\psi\,+ \\
& & \qquad\qquad\qquad
+\,\rho_\veps\,u_\veps^\perp\cdot\psi\,+\,\veps\,\nu\,\nabla u_\veps:\nabla\psi\biggr)\,dx\,dt\,=\,\veps\,\int_{\Omega}\rho_{0,\veps}\,u_{0,\veps}\cdot\psi(0)\,dx\,.
\end{eqnarray*}

By inequalities established in Subsection \ref{ss:bounds}, we know that $\bigl(\rho_\veps\,u_\veps\bigr)_\veps$ is uniformly bounded in e.g. $L^\infty_T(L^2)$,
and so are $\bigl(\rho_\veps\,u_\veps\otimes u_\veps\bigr)_\veps$ and $\bigl(\nabla u_\veps\bigr)_\veps$ respectively in $L^\infty_T(L^1)$ and in~$L^2_T(L^2)$.
Then, keeping in mind the assumptions on the initial data it is easy to see that, apart from the Coriolis operator, all the other terms in the previous relation
converge to $0$ in the limit $\veps\ra0$.

Therefore in   light of Lemma \ref{l:rho-u} above, we infer that
$$
\lim_{\veps\ra0}\,\int_0^T\int_\Omega\rho_\veps\,u_\veps^\perp\cdot\psi\,dx\,dt\;=\;\int_0^T\int_\Omega\rho\,u^\perp\cdot\psi\,dx\,dt\;=\;0
$$
for all test functions $\psi$ such that $\div\psi=0$. This property tells us that
\begin{equation} \label{eq:perp}
\rho\,u^\perp\,=\,\nabla\pi\,,
\end{equation}
for some suitable function $\pi$. Taking the $\rm curl$ of the previous relation, we immediately deduce that~$\div(\rho\,u)\,=\,0$.
Obviously,  the divergence-free condition on $u$, namely $\div u=0$, also has to be satisfied in the weak sense.

Let us now consider the mass equation. Thanks to the previous bounds and Lemma~\ref{l:rho-u}, we can pass to the limit in its weak formulation with no difficulty: we thus find
(still in the weak sense) the relation
$$
\d_t\rho\,+\,\mbox{div} (\rho\, u)\,=\,0\,,
$$
with initial condition $\rho_{|t=0}\,=\,\rho_0$. But the constraint $\div(\rho\,u)=0$ in $\mc D'$ forces us to have $\d_t\rho\,\equiv\,0$, and therefore
$$
\rho(t,x)\,\equiv\,\rho_0(x)\qquad\qquad\mbox{ for all }\quad (t,x)\,\in\,\R_+\times\Omega\,.
$$
To conclude, since now we have enough regularity, we can unravel the property $\div(\rho_0\,u)=0$ and write it as $u\cdot\nabla\rho_0\,=\,0$.
%\footnote{{\tt cet \'echange de d\'eriv\'ee a
%d\'ej\`a \'et\'e fait juste au-dessus}: but before it was in $\mc D'$; now we can put a true derivative on $\wtilde{\rho}_0$.}
Hence, the last sentence of the statement follows by noting that $u\cdot\nabla\rho_0$ belongs to $L^2_T(H^1)$.
\end{proof}

Before continuing our study, some remarks are in order.

\begin{rem} \label{r:constr-1}
For $\rho_0=1$, the only constraint on the limit velocity field is the relation $u^\perp\,=\,\nabla\pi$, which follows simply from the divergence-free condition.
On the other hand, it is easy to pass to the limit in the equation for the $r_\veps$'s (see also Section \ref{s:hom}).
\end{rem}

\begin{rem} \label{r:constr-cor}
In the fully non-homogeneous case, in the limit $\veps\ra0$, the Coriolis operator can be interpreted as a Lagrangian multiplier associated with the constraint
$\div\bigl(\rho_0\,u\bigr)\,=\,0$.
%
%Unluckily, we miss enough regularity on the density oscillations $\rho_\veps-\wtilde{\rho}_0$ to be able to rigorously prove the previous claim.
%Hence, we are obliged to resort to the vorticity formulation, where we can pass to the limit, but we lose the precise form of the target momentum equation. 
\end{rem}

\subsection{Further properties and remarks} \label{ss:further}

Before proving the convergence results, let us focus on the fully non-homogeneous case for a while, and establish further important properties
for the density function.

\subsubsection{The ``good unknown'' for the density} \label{sss:good}

As already pointed out at the end of Paragraph \ref{sss:b-dens}, the decomposition $\rho_\veps\,=\,\wtilde{\rho}_\veps+\veps r_\veps$,
where $\wtilde{\rho}_\veps$ and $r_\veps$ are obtained transporting respectively $\rho_0$ and $r_{0,\veps}$ by $u_\veps$, is not suitable for our analysis.

In view of Proposition \ref{p:constr}, the right \textsl{ansatz} to formulate is rather
$$
\rho_\veps(t,x)\;=\;\rho_0(x)\,+\,s_\veps(t,x) \, .
$$
First of all, the family $\bigl(s_\veps\bigr)_\veps$ is uniformly bounded in $L^\infty(\R_+\times\Omega)$, because both $\bigl(\rho_\veps\bigr)_\veps$ and $\rho_0$ are.
Moreover, it satisfies (in the weak sense) the equation
\begin{equation} \label{eq:s_e}
\d_ts_\veps\,+\,\mbox{div} ( s_\veps u_\veps ) \,=\,-\,\div\bigl(\rho_0\,u_\veps\bigr)\,=\,-\,u_\veps\cdot\nabla\rho_0\,,
\end{equation}
with initial datum $(s_\veps)_{|t=0}\,=\,\veps\,r_{0,\veps}\,\in\,L^2\cap L^\infty$. From \eqref{ub:u_L^p} and~(\ref{ub:u_L^2}) we infer that
$u_\veps\cdot\nabla\rho_0$ is uniformly bounded, for any $T>0$, in the space $L^2_T(L^2)$ and, by Sobolev embeddings, also in~$L^2_T(L^q)$ for all $2\leq q<+\infty$.
Therefore, since $\div u_\veps\,=\,0$, we deduce the uniform bounds
\begin{equation} \label{ub:s_L^p}
\bigl(s_\veps\bigr)_\veps\;\subset\;L^\infty\bigl([0,T];L^{2}(\Omega)\cap L^\infty(\Omega)\bigr)
\end{equation}
for all $T>0$. 

\begin{rem} \label{r:r-s}
Remark that $s_\veps\,\equiv\,\veps\,r_\veps$ in the case $\rho_0\equiv1$; similar uniform bounds have already been established above for the $r_\veps$'s, see \eqref{est:1_dens-L^p}
and Remark \ref{r:r_strong}.
\end{rem}

Furthermore, writing $ \d_ts_\veps\,=\,-\,\div\bigl(\rho_\veps\,u_\veps\bigr)$, we deduce that $\bigl(s_\veps\bigr)_\veps\,\subset\,W^{1,\infty}_T\bigl(H^{-1}(\Omega)\bigr)$.
Therefore, on the one hand, arguing exactly as in the last part of Paragraph \ref{sss:b-vel}, we can establish strong convergence properties for $s_\veps\,\longrightarrow\,0$
analogous to \eqref{cv:strong-rho} (keep in mind also Proposition \ref{p:constr}). 
On the other hand, we gather in particular the uniform embeddings
\begin{equation} \label{ub:s-C^g}
\bigl(s_\veps\bigr)_\veps\,\subset\,\mc C^{0,\g}\bigl([0,T];H^{-\g}(\Omega)\bigr)
\end{equation}
for all $\g\in[0,1]$.

Let us now define
\begin{equation} \label{def:svf}
\sigma_\veps\,:=\,\frac{1}{\veps}\,s_\veps\;,\qquad V_\veps\,:=\,\rho_\veps\,u_\veps\qquad\mbox{ and }\qquad
f_\veps\,:=\,-\,\div\bigl(\rho_\veps\,u_\veps\otimes u_\veps\bigr)\,+\,\nu\,\Delta u_\veps\,.
\end{equation}
We recall that, by uniform bounds, we have $\bigl(V_\veps\bigr)_\veps\,\subset\,L^\infty_T(L^2)\cap L^2_T(L^q)$ for all $2\leq q<+\infty$.
In the same way, we have $\rho_\veps\,u_\veps\otimes u_\veps$ uniformly bounded in $L^2_T(L^a)$, with $1/a\,=\,1/2+1/q$ and hence $a\in[1,2[\,$. 
By duality of Sobolev embeddings, we have that~$L^a\,\hra\,H^{-\alpha}$, where $\alpha\,=1-2/a'$ has to belong to $[0,1[\,$,  and $1/a+1/a'=1$. Hence we find
$\alpha\,=\,2\bigl(1/a\,-\,1/2\bigr)$ for all $a\in\,]1,2[\,$, that is to say we can write $\alpha\,=\,2/q$ with $2<q<+\infty$.
From this %and the fact that $u_\veps\in L^2_T(H^1)$ uniformly in $\veps$,
we deduce that
\begin{equation} \label{ub:f_e}
\bigl(f_\veps\bigr)_\veps\,\subset\,L^2\bigl([0,T];H^{-1-\alpha}(\Omega)\bigr)
\end{equation}
uniformly in $\veps\in\,]0,1]$, for any $\alpha\in]0,1[\,$.

\begin{rem} \label{r:ub-sigma_e}
By \eqref{ub:s_L^p}, we know that, for each $\veps$, $\sigma_\veps$ belongs to $L^\infty_T(L^{2}\cap L^\infty)$, but \textsl{a priori} we have \emph{no uniform bounds} at our disposal
for these quantities.
\end{rem}

With the previous notation, system \eqref{eq:dd-NSC} can be rewritten as
\begin{equation} \label{eq:waves}
\left\{\begin{array}{l}
        \veps\,\d_t\s_\veps\,+\,\div V_\veps\,=\,0 \\[1ex]
        \veps\,\d_tV_\veps\,+\,\nabla\Pi_\veps\,+\,V^\perp_\veps\,=\,\veps\,f_\veps\,,
       \end{array}
\right.
\end{equation}
with initial data $(\sigma_\veps)_{|t=0}\,=\,r_{0,\veps}\,\in\,L^2\cap L^\infty$ and $(V_\veps)_{|t=0}\,=\,m_{0,\veps}\,\in\,L^2$.
Taking the $\curl$ of the second equation in the sense of distributions, and then computing the difference with the first one, we arrive at the relation
\begin{equation} \label{eq:vort_form}
\d_t\bigl(\eta_\veps\,-\,\s_\veps\bigr)\,=\,\curl f_\veps\,,
\end{equation}
where we have set $\eta_\veps\,:=\,\curl(V_\veps)\,=\,\d_1V^2_\veps\,-\,\d_2V^1_\veps$. In view of \eqref{ub:f_e} and the properties on the initial data, we discover that
$\bigl(\eta_\veps-\s_\veps\bigr)_\veps$ is uniformly bounded in $\mc{C}^{0,1/2}_T(H^{-2-\alpha})$, and therefore, since by definition $\bigl(\eta_\veps\bigr)_\veps\subset L^\infty_T(H^{-1})$,
\begin{equation} \label{ub:sigma_e}
\bigl(\s_\veps\bigr)_\veps\,\subset\,L^\infty_T(H^{-2-\delta})
\end{equation}
uniformly in $\veps$, for all $\delta\in\,]0,1[\,$. In particular, there exists a $\sigma\,\in\,L^\infty_{\rm loc}\bigl(\R_+;H^{-2-\delta}(\Omega)\bigr)$ such
that~$\sigma_\veps\,\stackrel{*}{\rightharpoonup}\,\sigma$ in this space.

\begin{rem} \label{r:sigma_e}
Property \eqref{ub:sigma_e} is remarkable, because it establishes uniform bounds (even though in very negative spaces) for 
$\bigl(\rho_\veps-\rho_0\bigr)/\veps$. This uniform control is not obvious in the fully non-homogeneous case, if one just looks at the mass equation (pure transport);
to find it, we have used rather deeply the structure of our system.
\end{rem}

We are now going to exploit further this feature, and to establish a geometric property for the solutions to~\eqref{eq:dd-NSC}.

\subsubsection{A geometric property for the velocity fields} \label{sss:dispersion}

In this section we want to derive another remarkable feature of the dynamics.
Namely, dividing equation \eqref{eq:s_e} by $\veps$ we can formally write
$$
\d_t\s_\veps\,+\,\mbox{div} ( \s_\veps\, u_\veps )\,=\,-\,\frac{1}{\veps}\,u_\veps\cdot\nabla\rho_0\,.
$$
Thus, we somehow expect that $u_\veps\cdot\nabla\rho_0\,\longrightarrow\,0$ (in some sense) for $\veps\ra0$, with rate $O(\veps)$.
This is a strong geometric property for the family of $u_\veps$'s, which asymptotically align along the direction given by $\nabla^\perp\rho_0$ (keep also in mind
Proposition \ref{p:constr}), with some rate of convergence.
Unluckily, our bounds are not good enough to deduce that $\mbox{div} ( \s_\veps u_\veps )$ is uniformly controlled in a suitable space (keep in mind Remark \ref{r:ub-sigma_e} and
property \eqref{ub:sigma_e} above). Nonetheless, we are able to establish the following proposition, which will make this result a little more quantitative.
\begin{prop} \label{p:s-uniform}
Given~$\g\in\,]0,1[\,$, there exist
$$
0<\beta<\g\,,\qquad\g<k<1\,\qquad\mbox{ and }\qquad\theta\in\,]0,1[
$$
such that the uniform embeddings
$$ %\begin{equation} \label{ub:s-u}
\left(\frac{1}{\veps^\theta}\;s_\veps\right)_\veps\;\subset\;\mc C^{0,\beta}\bigl([0,T];H^{-k}(\Omega)\bigr)\qquad\mbox{ and }\qquad
\left(\frac{1}{\veps^\theta}\;s_\veps\,u_\veps\right)_\veps\;\subset\;L^2\bigl([0,T];H^{-k-\delta}(\Omega)\bigr)
$$ %\end{equation}
hold true for all $\delta>0$ arbitrarily small.

In particular, $\bigl(\veps^{-\theta}\,s_\veps\bigr)_\veps\,\ra\,0$ (strong convergence) in $L^\infty\bigl([0,T];H^{-k-\delta}_{\rm loc}(\Omega)\bigr)$
and $\veps^{-\theta}\,s_\veps\,u_\veps\,\rightharpoonup\,0$ (weak convergence) in $L^2\bigl([0,T];H^{-k-\delta}_{\rm loc}(\Omega)\bigr)$, for all $\delta>0$.
\end{prop}

\begin{proof}
We start the proof by recalling that, thanks to~\eqref{ub:s-C^g}, $\bigl(s_\veps\bigr)_\veps\,\subset\,\mc C^{0,\g}_T\bigl(H^{-\g}(\Omega)\bigr)$ for all $\g\in[0,1]$,
while, by~\eqref{ub:sigma_e},~$\bigl(\s_\veps\bigr)_\veps\,\subset\,L^\infty_T(H^{-2-\delta})$ for $\delta>0$ arbitrarily small, where $\s_\veps\,=\,s_\veps/\veps$.

Then, by interpolation of the previous uniform bounds, it easy to find (just estimate the difference $\|s_\veps(t)-s_\veps(\tau)\|_{H^{-k}}$ for $0\leq\tau<t\leq T$)
$$
\|s_\veps\|_{\mc C^{0,\beta}_T(H^{-k})}\,\leq\,C\,\|s_\veps\|_{L^\infty_T(H^{-2-\lam})}^\theta\;\|s_\veps\|^{1-\theta}_{\mc C^{0,\g}_T(H^{-\g})}\,,
$$
under the conditions $\beta\,=\,(1-\theta)\g$ and $k\,=\,\theta(2+\lam)\,+\,(1-\theta)\g$, for some $\theta\in\,]0,1[\,$.
For instance, one can make the explicit choices $\g=1/2$, $\lam=1/n$ for some $n\in\N$ and $k=3/4$, and determine consequently also the values of $\theta$ and $\beta$.
The important point here is that there exist $0<\beta<\g$ and $\g<k<1$ and a suitable corresponding $\theta\in\,]0,1[\,$, for which one has
$$
\veps^{1-\theta}\,\s_\veps\,=\, {\veps^{-\theta}}\,s_\veps\quad\in\;\mc C^{0,\beta}\bigl([0,T];H^{-k}(\Omega)\bigr)\,,
$$
and this property holds uniformly for $0<\veps\leq1$. In particular, by the Ascoli-Arzel\`a theorem, we deduce that this family is compact
in e.g. $L^\infty\bigl([0,T];H^{-k-\lam}_{\rm loc}(\Omega)\bigr)$, for all $\lam>0$, and hence it  converges strongly to $0$ in this space.

Next, let us remark that, by Corollary \ref{c:product} (iii) of the Appendix, the product is continuous on $H^{-k}\times H^1$, with values in $H^{-k-\delta}$, for arbitrarily small $\delta>0$.
By these considerations, we immediately see that the product $\veps^{-\theta}\,u_\veps\,s_\veps$ is well-defined and uniformly bounded in $L^2_T(H^{-k-\delta})$ with respect to $\veps$, for all
$\delta>0$:
$$ %\begin{equation} \label{ub:s-u}
\bigl({\veps^{-\theta}}\;s_\veps\,u_\veps \bigr)_\veps\;\subset\;L^2\bigl([0,T];H^{-k-\delta}(\Omega)\bigr)\,.
$$ %\end{equation}
Furthermore, $\veps^{-\theta}\,s_\veps\,u_\veps\,\rightharpoonup\,0$ in the space $L^2_T\bigl(H^{-k-\delta}_{\rm loc}(\Omega)\bigr)$, by the strong convergence property established above
on $\bigl(\veps^{-\theta}\,s_\veps\bigr)_\veps$.
\end{proof}

Let us end the discussion of this paragraph by pointing out further interesting properties, which however reveal to be not strong enough to be used in the analysis of the following sections.

Thanks to the properties established above, it makes sense (for instance in $\mc{D}'$) to write the equation
$$ %% \begin{equation} \label{eq:disp}
\d_t\left(\veps^{-\theta}\,s_\veps\right)\,+\,
\mbox{div} \left( \veps^{-\theta}\,s_\veps u_\veps \right)
 \,=\,-\, {\veps^{-\theta}}\,u_\veps\cdot\nabla\rho_0\,,
$$ %% \end{equation}
which has to be interpreted in a weak sense: for all test functions $\phi\in\mc{D}\bigl([0,T[\,\times\Omega\bigr)$, one has
\begin{equation} \label{eq:disp-w}
-\int^T_0\!\!\int_\Omega\veps^{-\theta}\,s_\veps\,\d_t\phi\,-\int^T_0\!\!\int_\Omega \veps^{-\theta}\,s_\veps\,u_\veps\cdot\nabla\phi\,=\,
\int_\Omega \veps^{1-\theta}\,r_{0,\veps}\,\phi(0)\,-\int^T_0\!\!\int_\Omega {\veps^{-\theta}}\,\rho_0\,u_\veps\cdot\nabla\phi\,.
\end{equation}
By density and the uniform bounds on $\veps^{-\theta}\,s_\veps$ established above, we notice that this expression makes sense and is continuous actually for all
$\phi\,\in\,W^{1,1}_T(H^{k+\delta+1})$, for any $\delta>0$: regularity in time comes from the first term (notice that we are losing the gain of $\beta$ derivatives
in time for~$\veps^{-\theta}\,s_\veps$), regularity in space (i.e. $k+\delta+1$) from the second one in the left-hand side.
Therefore we deduce the uniform bound 
\begin{equation} \label{ub:u-r_0} %%$$
\left({\veps^{-\theta}}\,u_\veps\cdot\nabla\rho_0\right)_\veps\;\subset\;W^{-1,\infty}_T(H^{-k})\,+\,L^2_T(H^{-k-\delta-1})\;\hookrightarrow\;
W^{-1,\infty}\bigl([0,T];H^{-k-\delta-1}(\Omega)\bigr)\,,
\end{equation}
for all $\delta>0$ arbitrarily small.
Actually, this term has to  converge weakly to $0$ in the previous space, because all the other terms in \eqref{eq:disp-w} go to $0$ in $\mc{D}'$.
So, we have gained a ``dispersion'' of order~$\veps^\theta$ (for some $\theta>0$) for the quantity $\bigl(u_\veps\cdot\nabla\rho_0\bigr)_\veps$.

We will not use directly this property in the convergence proof (for which we refer to Section~\ref{s:dens}), but rather the uniform controls provided by Proposition \ref{p:s-uniform}.

\bigskip

As it may appear clear from the discussion of this subsection, the main difficulty in the analysis when
$\rho_0$ is non-constant relies on the fact that we do not have at hand an explicit equation for the density oscillations $\s_\veps$,
nor do we have good enough uniform estimates for them: we are forced to use the vorticity associated with the velocity fields, and we get bounds in very rough spaces.

We will see in Section \ref{s:dens} how to handle these obstructions. For the time being, let us show the convergence in the slightly non-homogeneous case.

%%%%%%%%%%%%%%%%%%%%%%%%%%%%%%%%%%%%%%%%%%%%%%%%%%
%%%%%%%%%%%%%%%%%%%%%%%%%%%%%%%%%%%%%%%%%%%%%%%%%%
\section{Combining homogeneous and fast rotation limits} \label{s:hom}
%%%%%%%%%%%%%%%%%%%%%%%%%%%%%%%%%%%%%%%%%%%%%%%%%%
%%%%%%%%%%%%%%%%%%%%%%%%%%%%%%%%%%%%%%%%%%%%%%%%%%

We consider here the case $\rho_0\,\equiv\,1$, and we complete the proof of Theorem \ref{t:hom}: namely, we prove a convergence result in the weak formulation of
equations \eqref{eq:dd-NSC} and we identify the target system. We point out that, in this case, our fast rotation limit combines with a homogeneous limit,
because~$\rho_\veps\longrightarrow1$ in the limit $\veps\ra0$.

\subsection{Preliminary convergence properties} \label{ss:prelim-conv}
When $\rho_0\equiv1$, we have an explicit equation \eqref{eq:r_e} for the density oscillations $r_\veps:=(\rho_\veps-1)/\veps$, from which we have
deduced also suitable uniform bounds \eqref{est:1_dens-L^p}.
Moreover, in   light of Remark \ref{r:r_strong}, one can argue exactly as in
Lemma \ref{l:rho-u} and easily pass to the limit in the weak formulation of equation~\eqref{eq:r_e}. Thus, for $\veps$ going to $0$, we find
$$
\d_tr\,+\,
\mbox{div} (r\, u )
 \,=\,0\;,\qquad\qquad\mbox{ with }\qquad r_{|t=0}\,=\,r_0\,,
$$
where $r_0$ is the limit point of the sequence $\bigl(r_{0,\veps}\bigr)_\veps$ identified in \eqref{eq:conv-initial}.

Let us now consider the momentum equation. For $T>0$ fixed, let us consider a smooth function~$\psi\in\mc{D}\bigl([0,T[\,\times\Omega\bigr)$ such that $\div\psi=0$.
Using it in \eqref{eq:weak-momentum} we find
$$ %\begin{eqnarray*}
\int^T_0\!\!\int_{\Omega}\biggl(-\rho_\veps u_\veps\cdot\d_t\psi\,-\,\rho_\veps u_\veps\otimes u_\veps:\nabla\psi\,+\,\frac{1}{\veps}\,\rho_\veps u_\veps^\perp\cdot\psi\,+\,
\nu\nabla u_\veps:\nabla\psi\biggr)\,dx\,dt\,=\,\int_{\Omega}\rho_{0,\veps}u_{0,\veps}\cdot\psi(0)\,dx\,.
$$ %\end{eqnarray*}

Decomposing $\rho_\veps\,=\,1\,+\,\veps\,r_\veps$ and making use of the uniform bounds established in Subsection~\ref{ss:bounds}, it is an easy matter to pass to the limit in the $\d_t$ term
and in the viscosity term. Obviously, thanks to the assumptions on the initial data and especially properties \eqref{eq:conv-initial}, we have that
$\rho_{0,\veps}\,u_{0,\veps}\,\rightharpoonup\,u_0$ in $L^2$.

As for the Coriolis term, we can write
$$
\frac{1}{\veps}\int^T_0\int_\Omega\rho_\veps\,u_\veps^\perp\cdot\psi\,=\,\frac{1}{\veps}\int^T_0\int_\Omega u_\veps^\perp\cdot\psi\,+\,\int^T_0\int_\Omega r_\veps\,u_\veps^\perp\cdot\psi\,.
$$
Now, keeping in mind Remark \ref{r:constr-1}, we have that $\psi\,=\,\nabla^\perp\phi$: hence, the first term on the right-hand side identically vanishes (since $u_\veps$ is divergence-free),
while the second converges, in view of Lemma \ref{l:rho-u} again. In the end, we find
$$
\frac{1}{\veps}\int^T_0\int_\Omega\rho_\veps\,u_\veps^\perp\cdot\psi\,\longrightarrow\,\int^T_0\int_\Omega r\,u^\perp\cdot\psi\,.
$$
Notice that, alternatively, we could have used equation \eqref{eq:r_e} tested against $\phi$, to arrive at the relation
$\displaystyle \veps^{-1}\int^T_0\!\!\int_\Omega\rho_\veps u_\veps^\perp\cdot\psi\,=\,-\int^T_0\!\!\int_\Omega r_\veps\d_t\phi$.
This would have led us to exploit the vorticity formulation of the momentum equation, see also Remark \ref{r:1_vort}.

\medskip

It remains   to study the convergence of the convective term $\rho_\veps\,u_\veps\otimes u_\veps$.

We proceed in three steps: first of all, since we are in the slightly non-homogeneous regime,
we reduce our study to the constant density case. Then  we approximate the velocities by smooth vector fields, which
verify the same system up to small perturbations. These two steps are carried out in the next paragraph. Finally, in the third step  we use a compensated compactness argument, as in \cite{G-SR_2006}:
we perform integrations by parts (which are possible because we are dealing now with smooth vector fields) and, using the structure of the equations, we are finally able to pass to the limit.

\subsection{Some approximation lemmas} \label{ss:approx}
The following result is an obvious consequence of the fact that~$\rho_\veps\,=\,1\,+\,\veps\,r_\veps$,  along with the uniform bounds for $(r_\veps)_\veps$ in $L^\infty_T(L^\infty)$,
given by \eqref{est:1_dens-L^p}, combined with uniform bounds for $(u_\veps)_\veps$ in e.g. $L^\infty_T(L^2)$, see relation \eqref{ub:u_L^p} above.
\begin{lemma} \label{l:1_conv}
For any test function $\psi\,\in\,\mc{D}\bigl([0,T[\,\times\Omega\bigr)$, one has
$$
\lim_{\veps\ra0^+}\left|\int^T_0\int_\Omega\rho_\veps u_\veps\otimes u_\veps:\nabla\psi\,dx\,dt\,-\,\int^T_0\int_\Omega u_\veps\otimes u_\veps:\nabla\psi\,dx\,dt\right|\,=\,0\,.
$$
\end{lemma}
To perform the second step, for any $M\in\N$ we introduce the vector fields $u_{\veps,M}\,:=\,S_Mu_\veps$, where $S_M$ is the low frequency cut-off operator
defined by relation \eqref{eq:S_j} in the Appendix. Notice that, for any fixed $M$, we have
\begin{equation} \label{ub:u_e-M}
\bigl\|u_{\veps,M}\bigr\|_{L^\infty_T(H^k)\cap L^2_T(H^{k+1})}\,\leq\,C(T,k,M)\,,
\end{equation}
for any $T>0$ and any $k>0$. The constant $C(T,k,M)$ depends   on~$T$, $k$ and $M$, but   is uniform in $\veps>0$.
\begin{lemma} \label{l:uxu_conv}
For any test function $\psi\,\in\,\mc{D}\bigl([0,T[\,\times\Omega\bigr)$, one has
$$
\lim_{M\ra+\infty}\;\limsup_{\veps\ra0^+}\;
\left|\int^T_0\int_\Omega u_\veps\otimes u_\veps:\nabla\psi\,dx\,dt\,-\,\int^T_0\int_\Omega u_{\veps,M}\otimes u_{\veps,M}:\nabla\psi\,dx\,dt\right|\,=\,0\,.
$$
\end{lemma}

\begin{proof}
We start by estimating (see also \eqref{eq:LP-Sob} and Lemma \ref{l:Id-S} below), for any $\delta>0$, the difference
\begin{equation} \label{est:1-S_M-u}
%\left\|(\Id\,-\,S_M)u_\veps\right\|^2_{L^2_T(H^{1-\delta})}\,=\,\sum_{j\geq M}2^{2j(1-\delta)}\,\|\Delta_ju_\veps\|^2_{L^2_T(L^2)}
\left\|(\Id\,-\,S_M)u_\veps\right\|_{L^2_T(H^{1-\delta})}\,\leq\,2^{-\delta M}\,\left\|u_\veps\right\|_{L^2_T(H^1)}\,\leq\,
C\,2^{-\delta M}\,,
\end{equation}
for a constant $C>0$, uniform in $\veps$.
Therefore, keeping in mind also bounds \eqref{ub:u_L^p} for the family of velocity fields~$(u_\veps)_\veps$, we gather that both integrals
$$
I_1\,:=\,\int^T_0\int_\Omega (\Id-S_M)u_{\veps}\otimes u_{\veps}:\nabla\psi\qquad\mbox{ and }\qquad I_2\,:=\,\int^T_0\int_\Omega u_{\veps,M}\otimes (\Id-S_M)u_{\veps}:\nabla\psi 
$$
converge to $0$ for $M\ra+\infty$, uniformly with respect to the parameter $\veps\in\,]0,1]$.
This completes the proof of the statement.
\end{proof}

\subsection{The compensated compactness argument} \label{ss:comp-comp}
From now on, the argument is absolutely analogous to the one used in \cite{G-SR_2006}: let us sketch it for the reader's convenience.
Thanks to Lemmas \ref{l:1_conv} and \ref{l:uxu_conv}, we are reduced to studying, for any fixed~$M\in\N$, the convergence (with respect to $\veps$) of the integral
$$
-\int^T_0\int_\Omega u_{\veps,M}\otimes u_{\veps,M}:\nabla\psi\,dx\,dt\,=\,\int^T_0\int_\Omega\div\bigl(u_{\veps,M}\otimes u_{\veps,M}\bigr)\cdot\psi\,dx\,dt\,,
$$
where we have integrated by parts since now each $u_{\veps,M}$ is smooth in the space variable.
Remark that, since partial derivatives commute with operator $S_M$, we have $\div u_{\veps,M}=0$: then %%%we can write
\begin{equation} \label{eq:convective}
\div\bigl(u_{\veps,M}\otimes u_{\veps,M}\bigr)\,=\,u_{\veps,M}\cdot\nabla u_{\veps,M}\,=\,\frac{1}{2}\,\nabla\left|u_{\veps,M}\right|^2\,+\,\omega_{\veps,M}\,u_{\veps,M}^\perp\,,
\end{equation}
where $\omega_{\veps,M}\,:=\,\curl u_{\veps,M}$ is the vorticity of $u_{\veps,M}$. Notice that the former term in the last equality  vanishes identically
when tested against any test function $\psi$ such that   $\div\psi=0$.

As for the vorticity term, we start by reformulating the momentum equation in \eqref{eq:dd-NSC} in a similar way to what  was done in Paragraph \ref{sss:good}. More precisely, it is equivalent
to write (still in the weak sense)
$$
\veps\,\d_tV_\veps\,+\,\nabla\Pi_\veps\,+\,u_\veps^\perp\,=\,\veps\,f_\veps\,+\,\veps\,g_\veps\,,
$$
where $V_\veps$ and $f_\veps$ are defined in \eqref{def:svf} and we have set $g_\veps\,:=\,-\,r_\veps\,u_\veps^\perp$. Remark that the family $(g_\veps)_\veps$ is uniformly bounded in $L^\infty_T(L^2)$.
Applying the operator $S_M$ to the previous equation we find
$$
\veps\,\d_tV_{\veps,M}\,+\,\nabla\Pi_{\veps,M}\,+\,u_{\veps,M}^\perp\,=\,\veps\,\wtilde{f}_{\veps,M}\,, \qquad\qquad\mbox{ with }\qquad \wtilde{f}_\veps\,:=\,f_\veps+g_\veps\,.
$$
We   point out that, by uniform bounds and relation \eqref{ub:f_e}, for any $M$ and $k$ fixed and for all $T>0$, one has
$$
\left\|V_{\veps,M}\right\|_{L^\infty_T(H^k)}\,+\,\left\|\wtilde{f}_{\veps,M}\right\|_{L^2_T(H^k)}\,\leq\,C(T,k,M)\,,
$$
where the constant $C(T,k,M)$ does not depend on $\veps$. Furthermore, taking the $\curl$ of the   equation  and denoting $\eta_{\veps,M}\,=\,\curl V_{\veps,M}$, we get
$$
\d_t\eta_{\veps,M}\,=\,\curl\wtilde{f}_{\veps,M}\,.
$$
This last relation tells us that, for any fixed $M\in\N$ and $k\in\R$, the family $\bigl(\eta_{\veps,M}\bigr)_\veps$ is compact (in~$\veps$) in e.g. $L^\infty_T(H^k_{\rm loc})$,
and thus it converges strongly (up to extraction of a subsequence) to a tempered distribution $\eta_M$ in this space. But since
we already know   the convergence $V_\veps\,\rightharpoonup\,u$ in e.g. $L^\infty_T(L^2)$, it follows that~$\eta_\veps\,\rightharpoonup\,\omega\,=\,\curl u$ in e.g. $\mc D'$, hence~$\eta_M\,\equiv\,\omega_M$.

Finally writing~$\rho_\veps\,=\,1\,+\,\veps\,r_\veps$, we have  the identity
$$
\eta_{\veps,M}\,=\,\curl S_M(V_{\veps})\,=\,\omega_{\veps,M}\,+\,\veps\,\curl S_M(r_\veps\,u_\veps)\,,
$$
where  for any $M$ and $k$, the family $\bigl(\curl S_M(r_\veps\,u_\veps)\bigr)_\veps$ is uniformly bounded in $L^\infty_T(H^k)$.
From this relation and the above analysis, we deduce the strong convergence (still up to an extraction)
$$
\omega_{\veps,M}\,\longrightarrow\,\omega_M\qquad\qquad\mbox{ in }\qquad L^\infty_T(H^k_{\rm loc})
$$
for $\veps\ra0$. Using this property in equation \eqref{eq:convective}, we see that, up to passing to a suitable subsequence, for any $T>0$ and
any $M\in\N$, we have the convergence
$$
\lim_{\veps\ra0}\int^T_0\int_\Omega u_{\veps,M}\otimes u_{\veps,M}:\nabla\psi\,dx\,dt\,=\,\int^T_0\int_\Omega \omega_{M}\,u^\perp_{M}\cdot\psi\,dx\,dt\,.
$$
On the other hand, by uniform bounds and arguing as in \eqref{est:1-S_M-u}, we have the strong convergence% $\omega_M\,\longrightarrow\,\omega$ in $L^\infty_T(H^{-1})$ and
~$u_M\,\longrightarrow\,u$ in $L^2_T(H^{1})$ in the limit for $M\ra+\infty$ (by Lemma \ref{l:Id-S} below and Lebesgue dominated convergence theorem). Therefore, performing
equalities \eqref{eq:convective} backwards we can pass to the limit also in $M$. In the end, putting these properties together with Lemmas \ref{l:1_conv} and \ref{l:uxu_conv}, we have shown that
$$
\int^T_0\int_\Omega\rho_\veps u_\veps\otimes u_\veps:\nabla\psi\,dx\,dt\;\longrightarrow\;\int^T_0\int_\Omega u\otimes u:\nabla\psi\,dx\,dt
$$
for $\veps\ra0$, for all smooth divergence-free test functions $\psi$.

\medbreak
In order to complete the proof of Theorem \ref{t:hom}, it remains us to show that the whole sequence~$\bigl(\rho_\veps,u_\veps\bigr)_\veps$ converge. This fact is a straightforward consequence
of a uniqueness property for the target system, which is established in the next subsection.

\subsection{Study of the limit system} \label{ss:hom_limit}

The analysis we have just carried out provides us with the existence of weak solutions to the target system \eqref{eq:hom_lim}. In this subsection we want   to establish a uniqueness
result for equations \eqref{eq:hom_lim}. In particular, this result will imply the convergence of the whole sequence in Theorem~\ref{t:hom}.

Let us remark that, in performing stability estimates, a loss of one derivative appears, due to the hyperbolicity of the ``density'' equation
(i.e. the equation for $r$); now, propagating regularity for the gradient of $r$ asks for additional smoothness on $\bigl(r_0,u_0\bigr)$. In this respect, although $r$ and $u$ are coupled
\textsl{via} a zero order term, the system looks very much like a $2$-D incompressible density-dependent Navier-Stokes equations, for which uniqueness can be established for more regular initial
data (we refer for instance to Theorem 3.41 of \cite{B-C-D} and to Proposition 5.1 of \cite{D_2004}).

For simplicity of the exposition, we will prove stability estimates in energy spaces, the extension to more general functional frameworks going beyond the scope of the present paper.

\medbreak
The main result of this subsection reads as follows.
\begin{thm} \label{t:hom_r-u}
Given $\beta>0$, let $r_0\,\in\,H^{1+\beta}(\Omega)$ and $u_0\,\in\,H^1(\Omega)$ such that $\div u_0=0$.

Then there exists a unique weak solution $(r,u)$ to system \eqref{eq:hom_lim} with initial datum $(r_0,u_0)$, which satisfies the following properties:
\begin{enumerate}[(i)]
 \item the density $r$ belongs to $\mc C\bigl(\R_+;H^{1+\g}(\Omega)\bigr)$ for all $0\leq\g<\beta$; 
 \item for all $T>0$, $u $ belongs to $\mc C\bigl([0,T];H^1(\Omega)\bigr)\,\cap\,L^2\bigl([0,T];H^2(\Omega)\bigr)$.
% \item for all $T>0$ fixed, $\d_tu$ and $\nabla\Pi$ belong to the space $L^2\bigl([0,T];L^2(\Omega)\bigr)$.
\end{enumerate}
\end{thm}

The rest of this section is devoted to the proof of the previous statement. More precisely, we will focus   on the uniqueness part, the existence part being quite standard
(as a byproduct of Theorem \ref{t:hom} and propagation of regularity);  for the sake of completeness, we nevertheless show \textsl{a priori} estimates in Paragraph~\ref{sss:ex} before turning to the uniqueness result in Paragraph~\ref{sss:stabunique}.

\subsubsection{A priori estimates} \label{sss:ex}
In Theorem \ref{t:hom_r-u} we require   additional smoothness on the initial data, compared to what  can be expected from the singular limit problem.
We do not give details about the existence of solutions at that level of regularity but   focus   on how to get \textsl{a priori} estimates for system \eqref{eq:hom_lim}.

First of all, since by assumption $r_0\in L^2\cap L^\infty$ is transported by a divergence-free velocity field, for all $t\geq0$ we get
\begin{equation} \label{est:hom_r}
 \|r(t)\|_{L^q}\,=\,\|r_0\|_{L^q}\qquad\qquad\qquad\qquad \mbox{for all }\;q\,\in\,[2,+\infty]\,.
\end{equation}

As for $u$, we need to get higher regularity estimates: for this, we shall follow the same steps as for the proof of the existence of weak solutions to the non-homogeneous
Navier-Stokes equations at $H^1$ level of regularity. We refer to pages 31-32 of \cite{Lions_1} for details and original references of that result
(see also comments in Theorem 3.41 of \cite{B-C-D}, and Proposition 5.1 of \cite{D_2004}).
Actually, the situation here is simpler, since our equation for $u$ is a homogeneous Navier-Stokes equation:~$r$ appears just on the Coriolis
term, which can be treated as a forcing term at this level. As a consequence, we shall not attempt to prove refined estimates but only what is useful to infer the existence of a unique solution.

Before going on, let us point out that we will not keep track of the dependence of the various constants on the viscosity coefficient $\nu$, which is
positive and fixed.

First of all, if we multiply the second equation in \eqref{eq:hom_lim} by $u$, we integrate over $\Omega$ and then in time, we immediately get
\begin{equation} \label{est:hom_u}
\|u(t)\|^2_{L^2}\,+\,\nu\,\int^t_0\|\nabla u(\tau)\|^2_{L^2}\,d\tau\,\leq\,C\,\|u_0\|^2_{L^2}\,.
\end{equation}
%On the other hand, multiplying the same equation by $\d_tu$ and integrating over $\Omega$ give us
%$$
%\|\d_tu\|^2_{L^2}\,+\,\frac{\nu}{2}\,\frac{d}{dt}\,\|\nabla u\|^2_{L^2}\,\leq\,\left|\int_\Omega u\cdot\nabla u\,\cdot\,\d_tu\,dx\right|\,+\,\left|\int_\Omega r\,u^\perp\,\cdot\,\d_tu\,dx\right|\,.
%$$
%By H\"older   and Gagliardo-Nirenberg inequalities (see Proposition \ref{p:Gagl-Nir} in the appendix), it is easy to bound the terms in the right-hand side of the previous relation:
%\begin{eqnarray*}
%\left|\int_\Omega u\cdot\nabla u\,\cdot\,\d_tu\,dx\right| & \leq & C\,\|u\|_{L^4}\;\|\nabla u\|_{L^4}\;\|\d_tu\|_{L^2}\;\leq\;
%C\,\|u\|_{L^4}\;\|\nabla u\|^{1/2}_{L^2}\;\left\|\nabla^2u\right\|^{1/2}_{L^2}\;\|\d_tu\|_{L^2} \\
%\left|\int_\Omega r\,u^\perp\,\cdot\,\d_tu\,dx\right| & \leq & C\,\|r\|_{L^\infty}\;\|u\|_{L^2}\;\|\d_tu\|_{L^2}\,.
%\end{eqnarray*}
%Applying now Young's inequality and keeping in mind \eqref{est:hom_r} and \eqref{est:hom_u}, we arrive at the bound
%\begin{equation} \label{est:hom_du_1}
%\|\d_tu\|^2_{L^2}\,+\,\nu\,\frac{d}{dt}\,\|\nabla u\|^2_{L^2}\,\leq\,C\,\left(\|r_0\|^2_{L^\infty}\;\|u_0\|^2_{L^2}\,+\,\|u\|^2_{L^4}\;\|\nabla u\|_{L^2}\;\left\|\nabla^2u\right\|_{L^2}\right)\,,
%\end{equation}
Now  let us   multiply the equation by~$-\Delta u$. We find, thanks to H\"older's inequality,
$$
\begin{aligned}
\frac12\, \frac d{dt}\, \|\nabla u\|_{L^2}^2\, +\, \nu\, \|\Delta u\|_{L^2}^2\,& =\,- \int u \cdot \nabla u \cdot(-\Delta u)\, dx\, -\,  \int r u^\perp\cdot (-\Delta u)\, dx \\
& \leq\, \|u\|_{L^4}\, \|\nabla u\|_{L^4} \, \|\Delta u\|_{L^2}\, +\, \|r\|_{L^\infty}\,\|u\|_{L^2}\, \|\Delta u\|_{L^2} \,.
\end{aligned}
$$
By \eqref{est:hom_r} and \eqref{est:hom_u}, along with Gagliardo-Nirenberg's inequality, we infer that
$$
 \frac12\,\frac d{dt}\, \|\nabla u\|_{L^2}^2\, +\, \frac\nu2\, \|\Delta u\|_{L^2}^2\, \leq\, 
\|u_0\|^{1/2}_{L^2}\, \|\nabla u\|_{L^2}\,  \|\Delta u\|_{L^2}^{3/2}\, +\,C\, \|r_0\|^2_{L^\infty}\,\|u_0\|_{L^2}^2\,,
$$
hence finally
$$
   \frac d{dt}\, \|\nabla u\|_{L^2}^2\, +\, \nu\, \|\Delta u\|_{L^2}^2\,  \leq\, C\, \|u _0\|_{L^2}^2 \|\nabla u\|_{L^2}^4\,
 +\, \frac \nu2 \,  \|\Delta u\|_{L^2}^2\,+\, C\, \|r_0\|^2_{L^\infty}\,\|u_0\|_{L^2}^2\,.
$$
Gronwall's inequality gives, using \eqref{est:hom_u} again,
\begin{equation}\label{est:hom_du} 
 \|\nabla u(t)\|_{L^2}^2\, +\, \int_0^t  \|\Delta u(\tau)\|_{L^2}^2\, d\tau\, \leq\, C\, \Big(  \|\nabla u_0\|_{L^2}^2\,
 +\,T\,  \|r_0\|^2_{L^\infty}\,\|u_0\|_{L^2}^2 \Big)\,\exp\left(C\|u_0\|_{L^2}^4\right)\,.
\end{equation}

Now that we have obtained the property $u\,\in\,L^2\bigl([0,T];H^2\bigr)$, we can apply Proposition 5.2 of \cite{D_2004} (see also Theorem 3.33 of \cite{B-C-D}): we deduce that
$r\,\in\,\mc C\bigl([0,T];H^{1+\g}\bigr)$ for all $\g<\beta$, and it satisfies the estimate
\begin{align}
\|r(t)\|_{H^{1+\g}}\,&\leq\,C\,\exp\left(C\Big(\int^t_0\|\nabla u\|_{H^1}\,d\tau\Big)^{2}\right)\,\|r_0\|_{H^{1+\beta}} \label{est:hom_r_H} \\
&\leq\,C\,\exp\left(C\,T\,\Bigl(\|u_0\|^2_{L^2}+\left(\|\nabla u_0\|_{L^2}^2\,+\,T\,  \|r_0\|^2_{L^\infty}\,\|u_0\|_{L^2}^2\right)\,e^{C\|u_0\|_{L^2}^4}\Bigr)\right)\,
\|r_0\|_{H^{1+\beta}} \nonumber
\end{align}
for all $t\in[0,T]$, for a suitable constant $C$ depending on $\beta$ and $\g$.

\medbreak
These properties having been established, let us turn our attention to the \emph{uniqueness} issue.

\subsubsection{Stability estimates and uniqueness}\label{sss:stabunique}
Uniqueness of solutions to system \eqref{eq:hom_lim} is an immediate consequence of the next proposition, which provides a stability estimate.
In the proof, we limit ourselves to presenting the formal estimates, omitting a standard regularization procedure to make the computations rigorous.

\begin{prop} \label{p:stab}
Fix $\beta>0$. Consider two initial densities $r_{0,1}$ and $r_{0,2}$ in $H^{1+\beta}(\Omega)$, and two initial divergence-free velocity fields $u_{0,1}$ and $u_{0,2}$ in $H^1(\Omega)$.  
Let $(r_1,u_1)$ and $(r_2,u_2)$ be two solutions to system \eqref{eq:hom_lim} on $[0,T]\times\Omega$, related to the initial data $(r_{0,1},u_{0,1})$ and $(r_{0,2},u_{0,2})$
respectively. Define $\de r\,:=\,r_1-r_2$ and~$\de u\,:=\,u_1-u_2$, and let $\de r_0$ and $\de u_0$ be the same quantities computed on the initial data.

Then there exists %%%a universal constant $C>0$ and
a constant $C_0(T)$, depending just on the norms of the initial data and on the fixed time $T>0$, such that the following estimates hold true
for all $t\in[0,T]$:
$$\|\de r(t)\|^2_{L^2}\,+\,\|\de u(t)\|^2_{H^1} \,\leq\,
C_0(T)\;e^{C_0(T)}\;\bigl(\|\de r_0\|^2_{L^2}\,+\,\|\de u_0\|^2_{H^1}\bigr)\,.
$$
\end{prop}

\begin{rem} \label{r:stab}
The $H^1$ assumption on the initial velocity field is needed in order to control $\nabla r$. In fact, this term arises from stability estimates for the mass equation; it represents a loss of one
derivative due to hyperbolicity of this equation.
Imposing higher smoothness on $u$ allows us to avoid the use of the $L^\infty$ norm of $\nabla r$, which has no chance of being bounded at this level of regularity (see inequality \eqref{est:delta-r}
below for further details). Note that the statement and its proof are far from being optimal but are enough for our purposes.
\end{rem}

\begin{proof}[Proof of Proposition~{\rm\ref{p:stab}}]
Let us start by considering the transport equation for the density: taking the difference of the equation for $r_1$ by the equation for $r_2$, we find that $\de r$ satisfies
$$
\d_t\de r\,+\,u_2\cdot\nabla\de r\,=\,-\,\de u\cdot\nabla r_1\,.
%\d_t\de r\,+\,\div(\de r\,u_2)\,=\,-\,\div(r_1\,\de u)\,.
$$
We now take the scalar product in $L^2$ of this equation with $\de r$ and integrate in time: easy computations lead   to the estimate
$$
\|\de r(t)\|^2_{L^2}\,\leq\,C\,\left(\|\de r_0\|^2_{L^2}\,+\,\int^t_0\left\|\de u\cdot\nabla r_1\right\|^2_{L^2}\,d\tau\,+\,\int^t_0\left\|\de r\right\|^2_{L^2}\,d\tau\right)\,.
$$
Thanks to Corollary~\ref{c:product} (iii) we have
$$ \left\|\de u\cdot\nabla r_1\right\|_{L^2} \,  \leq \,\|\de u\|_{H^1}\,\|r_1\|_{H^{1+\beta/2}} \,;
$$
so we infer, for all~$0\leq t \leq T$, that
$$
\begin{aligned}
\|\de r(t)\|^2_{L^2}\,&\leq\,C\|\de r_0\|^2_{L^2}\,+\,C\,\int^t_0\|\de u\|^2_{H^1}\,\|r_1\|^2_{H^{1+\beta/2}}\,d\tau\,+\,C\int^t_0\left\|\de r\right\|^2_{L^2}\,d\tau \\
\,&\leq\,C\,\|\de r_0\|^2_{L^2}\,+\,C_0(T)\,\int^t_0\|\de u\|^2_{H^1} \,d\tau\,+\,C\int^t_0\left\|\de r\right\|^2_{L^2}\,d\tau\,,
\end{aligned}
$$
thanks to~\eqref{est:hom_r_H}, where $C_0(T)$ is a constant depending  on   norms of the initial data and on     time~$T>0$, and can change from line to line.
By Gronwall's inequality  we obtain that,  for all~$0\leq t \leq T$,
\begin{equation} \label{est:delta-r}
\|\de r(t)\|^2_{L^2}\,\leq\,  \Big( \|\de r_0\|^2_{L^2}\,+\,C_0(T)\,\int^t_0\|\de u\|^2_{H^1}\,d\tau \Big) \,\exp(C\,T)\,.
\end{equation}
Let us focus now on the momentum equations: easy computations show that $\de u$ solves
\begin{equation} \label{eq:delta-u}
\d_t\de u\,+\,u_2\cdot\nabla\de u\,+\,\de u\cdot\nabla u_1\,+\,\nabla\de\Pi\,+\,r_2\,\de u^\perp\,+\,\de r\,u_1^\perp\,-\,\nu\,\Delta\de u\,=\,0\,,
\end{equation}
where~$\de\Pi:=\Pi_1-\Pi_2$.
An~$L^2$ energy estimate provides
$$
\frac{1}{2}\,\frac{d}{dt}\|\de u\|_{L^2}^2 \,+\,\nu \|\nabla\de u\|_{L^2}^2 \,=\,-\,\int\de u\cdot\nabla u_1\,\cdot\,\de u\,dx\,-\,\int\de r\,u_1^\perp\,\cdot\,\de u\,dx\,.
$$
The terms on the right-hand side can be bounded in the following way:
\begin{eqnarray*}
\left|\int\de r\,u_1^\perp\,\cdot\,\de u \, dx \right| & \leq & C\,\|\de r\|_{L^2}\;\|u_1\|_{L^\infty}\;\|\de u\|_{L^2} \leq C\,\|u_1\|_{L^\infty}\; \left(\|\de r\|_{L^2}^2\;+ \|\de u\|_{L^2}^2\right)\\
\left|\int\de u\cdot\nabla u_1\,\cdot\,\de u \, dx \right| & \leq &  C\,\|\de u\|_{L^4}\;\|u_1\|_{L^4}\;\|\nabla\de u\|_{L^2}\leq C\,\;\|\de u\|^{1/2}_{L^2}\|\nabla\de u\|^{3/2}_{L^2}\;\|u_1\|_{L^4}\,,
\end{eqnarray*}
where, for the latter term, we have performed an integration by parts and we have used also Gagliardo-Nirenberg inequality. By Young's inequality, after an integration in time, it is easy to arrive
at the estimate
\begin{eqnarray*}
& & \hspace{-0.5cm} \|\de u(t)\|^2_{L^2}\,+\, \nu\int^t_0\|\nabla\de u(\tau)\|^2_{L^2}\,d\tau\,  \\
& & \qquad\qquad\leq\,\|\de u_0\|^2_{L^2}\,+\,C\int^t_0\|\de u\|^2_{L^2}\,\|u_1\|^4_{L^4}\,d\tau\,+\,C\int^t_0 (\|\de r\|_{L^2}^2\;+ \|\de u\|_{L^2}^2)\,\|u_1\| _{L^\infty}\,d\tau\,. \nonumber
\end{eqnarray*}
Gronwall's lemma provides
$$
\begin{aligned}
 \|\de u(t)\|^2_{L^2}\,+\,  \int^t_0\|\nabla\de u(\tau)\|^2_{L^2}\,d\tau &\leq  C\, \Big( \|\de u_0\|^2_{L^2}+
  \int^t_0\|\de r\|_{L^2}^2 \,\|u_1\| _{L^\infty}\,d\tau
 \Big)\\
 &\qquad\times\,\exp \left( C\int_0^T \big( \|u_1\|^4_{L^4}  + \|u_1\| _{L^\infty}\big)\,d\tau\right) \,.
\end{aligned}
 $$
 Hence, using the fact that
 $$
 \|u_1\| _{L^\infty} \leq C ( \|u_1\| _{L^2} +  \|u_1\| _{H^2} )\,,
 $$
 we infer
 $$
\begin{aligned}
 \|\de u(t)\|^2_{L^2}\,+\,  \int^t_0\|\nabla\de u(\tau)\|^2_{L^2}\,d\tau &\leq    C \Big( \|\de u_0\|^2_{L^2}+   \int^t_0\|\de r\|_{L^2}^2( \|u_1\| _{L^2} +  \|u_1\| _{H^2} ) \,d\tau
 \Big)\,e^{C_0(T)}\,.
 \end{aligned}
 $$
 It remains to plug that inequality into~(\ref{est:delta-r}) to find
$$
\begin{aligned}
&\|\de r(t)\|^2_{L^2}
+ \|\de u(t)\|^2_{L^2} +    \int^t_0\|\nabla\de u(\tau)\|^2_{L^2}\,d\tau   \\
&\quad \leq  \,C_0(T)\,\Big( \|\de u_0\|^2_{L^2}+ \|\de r_0\|^2_{L^2}\,+\, \int^t_0  \|\de r\|_{L^2}^2( \|u_1\| _{L^2} +  \|u_1\| _{H^2} ) \,d\tau\,+\,\int^t_0 \| \delta u \|^2_{L^2} d\tau
 \Big)\,e^{C_0(T)}
\,.
\end{aligned}
 $$
 Gronwall's lemma ends the proof of the proposition.
\end{proof}

%%%%%%%%%%%%%%%%%%%%%%%%%%%%%%%%%%%%%%%%%%%%%%%%%%
%%%%%%%%%%%%%%%%%%%%%%%%%%%%%%%%%%%%%%%%%%%%%%%%%%
\section{The fully non-homogeneous case} \label{s:dens}
%%%%%%%%%%%%%%%%%%%%%%%%%%%%%%%%%%%%%%%%%%%%%%%%%%
%%%%%%%%%%%%%%%%%%%%%%%%%%%%%%%%%%%%%%%%%%%%%%%%%%

In this section we tackle the convergence for $\veps\ra0$ in the case of a non-constant target density. Subsections \ref{ss:vort_weak} to \ref{ss:limit} %%, \ref{ss:convective} and \ref{ss:limit}
are devoted to the proof of Theorem \ref{t:dens}. In Subsection \ref{ss:dens-full} we present a conditional result, where we are able to show convergence to
(a slightly modified version of) the full system \eqref{eq:dens_lim-full}.

\medbreak
We start by noticing that passing to the limit in the mass equation involves no special difficulty, and   can be done as in the proof of Proposition \ref{p:constr}.
So we have to show convergence in the momentum equation, i.e. in the relation
$$ %\begin{eqnarray*}
\int^T_0\!\!\int_{\Omega}\biggl(-\rho_\veps u_\veps\cdot\d_t\psi\,-\,\rho_\veps u_\veps\otimes u_\veps:\nabla\psi\,+\,\frac{1}{\veps}\,\rho_\veps u_\veps^\perp\cdot\psi\,+\,
\nu\nabla u_\veps:\nabla\psi\biggr)\,dx\,dt\,=\,\int_{\Omega}m_{0,\veps}\cdot\psi(0)\,dx\,,
$$ %\end{eqnarray*}
where $\psi$ is a smooth divergence-free test function, compactly supported in $[0,T[\,\times\Omega$.

Once again, thanks to uniform bounds and Lemma \ref{l:rho-u}, it is easy to perform the limit $\veps\ra0$ in the time derivative and viscous terms.
On the contrary, the Coriolis term involves some complications: let us focus on it for a while.

\subsection{Passing to the vorticity formulation} \label{ss:vort_weak}

For convenience, let us adopt the same notation as in \eqref{def:svf} and write the previous expression under the form
\begin{equation} \label{eq:V_weak}
\int^T_0\!\!\int_{\Omega}\biggl(-V_\veps\cdot\d_t\psi\,+\,\frac{1}{\veps}\,V_\veps^\perp\cdot\psi\biggr)\,dx\,dt\,=\,
\int^T_0\lan f_\veps,\psi\ran\,dt\,+\,\int_{\Omega}V_{0,\veps}\cdot\psi(0)\,dx\,,
\end{equation}
where we denote by $\lan\cdot\,,\,\cdot\ran$ the duality product in e.g. $H^{-1-\alpha}\times H^{1+\alpha}$ (recall relation \eqref{ub:f_e} above)
and we have set $V_{0,\veps}\,:=\,m_{0,\veps}$.
As before, $\psi$ is a smooth test function having zero divergence.

Our concern here is to pass to the limit in the rotation term $\veps^{-1}\,V_\veps^\perp$. To do so, owing to the fact that $\div\psi=0$, let us write
$\psi\,=\,\nabla^\perp\phi$, for some smooth scalar function $\phi\in\mc D\bigl([0,T[\,\times\Omega\bigr)$.
Using   the mass equation tested against~$\phi$, we get (keep in mind \eqref{eq:weak-mass} and \eqref{def:svf} above)
$$
\frac{1}{\veps}\int^T_0\int_{\Omega} V_\veps^\perp\cdot\psi\,=\,\frac{1}{\veps}\int^T_0\int_\Omega V_\veps\cdot\nabla\phi\,=\,
-\int^T_0\int_\Omega\sigma_\veps\,\d_t\phi\,-\,\int_\Omega r_{0,\veps}\,\phi(0)\,.
$$

On the one hand, this trick seems to us the only way to avoid the singularity of the Coriolis term. On the other hand, it forces us to pass to the vorticity formulation;
hence, we rewrite \eqref{eq:V_weak} in the following form:
\begin{equation} \label{eq:vort_weak}
-\int^T_0\!\!\int_\Omega V_\veps\cdot\d_t\nabla^\perp\phi\,-\,\int^T_0\!\!\int_\Omega\s_\veps\,\d_t\phi\,=\,\int^T_0\lan f_\veps,\nabla^\perp\phi\ran\,+\,
\int_\Omega V_{0,\veps}\cdot\nabla^\perp\phi(0)\,+\,\int_\Omega r_{0,\veps}\,\phi(0)\,.
\end{equation}
Notice that we have done nothing but finding again equation \eqref{eq:vort_form}.

Our goal becomes then to pass to the limit in equation \eqref{eq:vort_weak}.
Recalling the definition of $f_\veps$ given in \eqref{def:svf} and the discussion at the beginning of the present section, the only difficulty
relies in proving convergence in the convective term: so  let us focus our attention on it.

\subsection{Handling the convective term} \label{ss:convective}

In   light of the previous discussion,   it remains to pass to the limit in the convective term, namely in relation
$$
\int^T_0\int_\Omega\rho_\veps\,u_\veps\otimes u_\veps\,:\,\nabla\nabla^\perp\phi\,dx\,dt\,.
$$
The strategy is similar to the one adopted above for slightly non-homogeneous fluids, see in particular Subsections \ref{ss:approx} and \ref{ss:comp-comp}.

First of all, let us write down the system of wave equations, which governs the propagation of oscillations in the dynamics. We could call them \emph{Rossby waves}:
this terminology is typically used when there are variations of the rotation axis (see e.g. \cite{CR}), but  in our context having a non-constant
$\rho_0$ produces very similar effects.

\subsubsection{Description of Rossby waves and regularization} \label{sss:rossby}

With the notation introduced in \eqref{def:svf}, we can write system \eqref{eq:dd-NSC} in the form of a wave equation~\eqref{eq:waves}, which we recall here for convenience:
\begin{equation} \label{eq:waves_2}
\left\{\begin{array}{l}
        \veps\,\d_t\s_\veps\,+\,\div V_\veps\,=\,0 \\[1ex]
        \veps\,\d_tV_\veps\,+\,\nabla\Pi_\veps\,+\,V^\perp_\veps\,=\,\veps\,f_\veps\,.
       \end{array}
\right.
\end{equation}
Recall also  the uniform bounds \eqref{ub:f_e} for $(f_\veps)_\veps$. Applying the $\curl$ operator to the second equation, we find
\begin{equation} \label{eq:waves_vort}
\left\{\begin{array}{l}
        \veps\,\d_t\s_\veps\,+\,\div V_\veps\,=\,0 \\[1ex]
        \veps\,\d_t\eta_\veps\,+\,\div V_\veps\,=\,\veps\,\curl f_\veps\,,
       \end{array}
\right.
\end{equation}
where $\eta_\veps\,=\,\curl V_\veps$ as in Subsection \ref{ss:further} above.

As in Subsection \ref{ss:approx}, we now fix $M\in\N$ and, denoting by $S_M$ the low frequency cut-off operator of a Littlewood-Paley decomposition introduced in \eqref{eq:S_j}
in the Appendix, we define
$$
\sigma_{\veps,M}\,:=\,S_M\s_\veps\;,\qquad V_{\veps,M}\,:=\,S_MV_\veps\qquad\mbox{ and }\qquad
\eta_{\veps,M}\,:=\,\curl V_{\veps,M}\,=\,S_M\eta_\veps\,.
$$
Notice that all these quantities are smooth with respect to the $x$ variable. For later use, let us immediately introduce also the smooth functions
$$
u_{\veps,M}\,:=\,S_Mu_\veps\qquad\qquad\mbox{ and }\qquad\qquad \omega_{\veps,M}\,:=\,\curl u_{\veps,M}\,=\,S_M\omega_\veps\,.
$$
Let us point out that, analogously to \eqref{ub:u_e-M}, by uniform bounds we get
$$ %%\begin{equation} \label{ub_dens:u_e-M}
\bigl\|u_{\veps,M}\bigr\|_{L^2_T(H^{k})}\,\leq\,C(T,k,M)
$$ %%\end{equation}
for all $k\geq1$, for some constant $C(T,k,M)$ just depending on the quantities in the brackets; in particular, for $k=1$ we have
\begin{equation} \label{ub_dens:u_2}
\bigl\|u_{\veps,M}\bigr\|_{L^2_T(H^{1})}\,\leq\,C(T)\,.
\end{equation}

The following regularization result holds true. Its proof is standard, and hence omitted. We limit ourselves to noticing here that the convergence properties \eqref{reg:convergence}
are straightforward consequences of the uniform bound \eqref{ub:sigma_e} for $\bigl(\sigma_\veps\bigr)_\veps$ and the embedding $\bigl(\eta_\veps\bigr)_\veps\,\subset\,L^\infty_T(H^{-1})$.
\begin{prop} \label{p:regular}
For any fixed time $T>0$, the following convergence properties hold, in the limit   $M\longrightarrow+\infty$:
\begin{equation} \label{reg:convergence}
\begin{cases}
\; \sup_{\veps>0}\left\|\s_\veps\,-\,\s_{\veps,M}\right\|_{L^\infty_T(H^{-s})}\,\longrightarrow\,0
\qquad\qquad\forall\; s>2 \\[1ex]
\; \sup_{\veps>0}\left\|\eta_\veps\,-\,\eta_{\veps,M}\right\|_{L^\infty_T(H^{-s})}\,\longrightarrow\,0
\qquad\qquad\forall\; s>1\,.
\end{cases}
\end{equation}
Moreover, for any $M>0$, the couple $\bigl(\s_{\veps,M}\,,\,\eta_{\veps,M}\bigr)$ satisfies the wave equation
\begin{equation} \label{reg:approx-w}
\begin{cases}
\; \veps\,\d_t\s_{\veps,M}\,+\,\div\,V_{\veps,M}\,=\,0 \\[1ex]
\; \veps\,\d_t\eta_{\veps,M}\,+\,\div\,V_{\veps,M}\,=\,\veps\,\curl f_{\veps,M}\,,
\end{cases}
\end{equation}
where $\bigl(f_{\veps,M}\bigr)_{\veps}$ is a family of smooth functions satisfying
\begin{equation} \label{reg:source}
\sup_{\veps>0}\left\|f_{\veps,M}\right\|_{L^2_T(H^{s})}\,\leq\,C(s,M,T)
\qquad\qquad\forall\; s\geq0\,,
\end{equation}
for suitable constants $C(s,M,T)$ depending only on the fixed~$s\geq0$, $M>0$ and~$T>0$.
\end{prop}

\begin{comment}
\begin{proof}
Keeping in mind the characterization of $H^s$ spaces in terms of Littlewood-Paley decomposition (see Appendix \ref{app:LP}),
properties \eqref{reg:convergence} are straightforward consequences of the uniform bounds established in Subsections \ref{ss:bounds} and \ref{ss:further}.

Next, applying operator $S_M$ to \eqref{eq:waves_vort} above immediately gives us system \eqref{reg:approx-w},
where we have set $f_{\veps,M}\,:=\,S_Mf_\veps$.
By this definition and \eqref{ub:f_e}, the uniform bounds of \eqref{reg:source} easily follow.
\end{proof}
\end{comment}

Let us conclude this part by proving an important statement, which enables us to compare the two velocity fields~$V_{\veps,M}$ and $u_{\veps,M}$.
\begin{prop} \label{p:u-V}
For all $M\in\N$ and all $\veps\in\,]0,1]$, the following equality holds true:
$$
V_{\veps,M}\,=\,\rho_0\,u_{\veps,M}\,+\,\veps^\theta\,\z_{\veps,M}\,+\,h_{\veps,M}\,,
$$
where $0<\theta<1$ is the exponent fixed in Proposition {\rm \ref{p:s-uniform}}. Moreover, the families $\bigl(\z_{\veps,M}\bigr)_{\veps}$ and~$\bigl(h_{\veps,M}\bigr)_{\veps}$
satisfy the uniform bounds
$$
\begin{cases}
\; \sup_{\veps>0}\left\|\z_{\veps,M}\right\|_{L^2_T(H^{s} )}\,\leq\,C(s,M,T)
\qquad\qquad\forall\; s\geq0 \\[1ex]
\; \sup_{\veps>0}\left\|h_{\veps,M}\right\|_{L^2_T(H^1 )}\,\leq\,C_0(T)\,2^{-M}\,,
%\qquad\qquad\mbox{ for }\quad M\ra+\infty\,,
\end{cases}
$$
for any time $T>0$, where the constants $C(s,M)$ and $C_0$ depend only on   norms of the initial data, on   time $T$ and, respectively,
on the fixed $s$ and $M$, and on the norm $\left\|\rho_0\right\|_{W^{2,\infty}}$.
\end{prop}

\begin{proof}
By definition, we can write
\begin{equation} \label{eq:v-dec}
V_{\veps,M}\,=\,S_M\bigl(\rho_\veps\,u_\veps\bigr)\,=\,\veps^\theta\,S_M\bigl(\veps^{-\theta}\,s_\veps\,u_\veps\bigr)\,+\,S_M\bigl(\rho_0\,u_\veps\bigr)\,,
\end{equation}
where we recall that $s_\veps\,=\,\rho_\veps-\rho_0$ expresses oscillations of the density with respect to the target profile. Thanks to Proposition \ref{p:s-uniform}, 
we deduce that, for any fixed $M\in\N$, the family  
$$
\z_{\veps,M}\,:=\,\frac{1}{\veps^\theta}\;S_M\bigl(s_\veps\,u_\veps\bigr) $$
is bounded in~$L^2_T(H^s)$, uniformly with respect to $\veps\in\,]0,1]$, for any $s\geq0$.

On the other hand, denoting by $[\mc{P},\mc{Q}]$ the commutator between two operators $\mc P$ and $\mc Q$, the following relation holds true:
$$
S_M\bigl(\rho_0\,u_\veps\bigr)\,=\,\rho_0\,u_{\veps,M}\,+\,\bigl[S_M,\rho_0\bigr]u_\veps\,.
$$
So, let us set $h_{\veps,M}\,:=\,\bigl[S_M,\rho_0\bigr]u_\veps$ and estimate its $H^1$ norm. First of all, by Lemma \ref{l:commut}
and uniform bounds \eqref{ub:u_L^2}, we immediately gather
$$
\sup_{\veps>0}\left\|h_{\veps,M}\right\|_{L^2_T(L^2)}\,\leq\,C\,2^{-M}\,\left\|\nabla\rho_0\right\|_{L^\infty}\,\|u_{\veps}\|_{L^2_T(L^2)}\,\leq\,
C'\,2^{-M}\,.
$$
Furthermore, from the equality $\d_jh_{\veps,M}\,=\,\bigl[S_M,\rho_0\bigr]\d_ju_\veps\,+\,\bigl[S_M,\d_j\rho_0\bigr]u_\veps$, the combination of Lemma \ref{l:commut} with
\eqref{ub:u_L^p} and \eqref{ub:u_L^2} implies
$$
\sup_{\veps>0}\left\|\d_jh_{\veps,M}\right\|_{L^2_T(L^2)}\,\leq\,C\,2^{-M}\,\Bigl(\left\|\nabla\rho_0\right\|_{L^\infty}\,\|\nabla u_{\veps}\|_{L^2_T(L^2)}\,+\,
\left\|\nabla^2\rho_0\right\|_{L^\infty}\,\|u_{\veps}\|_{L^2_T(L^2)}\Bigr) \leq C\,2^{-M}\\.$$
 The proposition is now completely proved.
\end{proof}

From Proposition \ref{p:u-V} we immediately infer the next statement.  Notice that we propose two (related but different) decompositions for the approximate vorticities
$\eta_{\veps,M}$: they will both be  useful in our analysis.
\begin{coroll} \label{c:rot_u-V}
For all $M\in\N$ and all $\veps\in\,]0,1]$, the following equalities hold true:
$$
\begin{cases}
\; \eta_{\veps,M}\,&=\,\eta_{\veps,M}^{(1)}\,+\,\veps^\theta\,\eta_{\veps,M}^{(2)} \\[1ex]
&=\,\rho_0\,\omega_{\veps,M}\,+\,u_{\veps,M}\cdot\nabla^\perp\rho_0\,+\,\veps^\theta\,\curl\z_{\veps,M}\,+\,\curl h_{\veps,M} \\[1ex]
\; \div V_{\veps,M}\,&=\,u_{\veps,M}\cdot\nabla\rho_0\,+\,\veps^\theta\,\div\z_{\veps,M}\,+\,\div h_{\veps,M}\,,
\end{cases}
$$
where the families~$\big(\eta_{\veps,M}^{(1)}\big)_\veps$ and~$\big(\eta_{\veps,M}^{(2)}\big)_\veps$ satisfy the following estimates:
$$
\sup_{M\in\N}\;\sup_{\veps>0}\big\|\eta_{\veps,M}^{(1)}\big\|_{L^2_T(L^2)}\,\leq\,C\quad \mbox{ and }\quad
\sup_{\veps>0}\big\|\eta^{(2)}_{\veps,M}\big\|_{L^2_T(H^{s})}\,\leq\,C(s,M)\,,
$$
for any given~$s\geq0$ fixed and where~$\bigl(\z_{\veps,M}\bigr)_{\veps}$ and $\bigl(h_{\veps,M}\bigr)_{\veps}$ are defined in Proposition~{\rm\ref{p:u-V}} and, in particular,   satisfy
the same bounds established there.   
\end{coroll}

\begin{proof}
The previous statement is a straightforward consequence of Proposition \ref{p:u-V}, and hence omitted. We just notice here that the former decomposition for
$\eta_{\veps,M}$ comes directly from relation~\eqref{eq:v-dec}, by setting
$$
\eta_{\veps,M}^{(1)}\,:=\,\curl S_M\bigl(\rho_0\,u_\veps\bigr)\qquad\qquad\mbox{ and }\qquad\qquad \eta_{\veps,M}^{(2)}\,:=\,\curl S_M\left(\veps^{-\theta}\,s_\veps\,u_\veps\right)\,.
$$
The uniform bounds for these quantities then derive from \eqref{ub:u_L^p}, \eqref{ub:u_L^2} and Proposition \ref{p:s-uniform}.
\end{proof}

\subsubsection{Approximation and convergence} \label{sss:dens-cc}

In this paragraph we state and prove some approximation results, in the same spirit as the ones given in Subsection \ref{ss:approx} above, and we take the limit
in the convective term by compensated compactness arguments.

Let us start by establishing the counterpart of Lemma \ref{l:1_conv} in the case of a non-constant density profile $\rho_0$.

\begin{lemma} \label{l:dens_conv}
For any test function $\psi\,\in\,\mc{D}\bigl([0,T[\,\times\Omega\bigr)$, one has
$$
\lim_{\veps\ra0^+}\left|\int^T_0\int_\Omega\rho_\veps u_\veps\otimes u_\veps:\nabla\psi\,dx\,dt\,-\,\int^T_0\int_\Omega\rho_0\,u_\veps\otimes u_\veps:\nabla\psi\,dx\,dt\right|\,=\,0\,.
$$
\end{lemma}

\begin{proof}
We decompose the density, according to our \textsl{ansatz}, as $\rho_\veps\,=\,\rho_0\,+\,s_\veps$.
Recall that, by \eqref{cv:strong-rho} (keep in mind also \eqref{ub:s-C^g} and Proposition \ref{p:s-uniform}), we know that, up to passing
to subsequences, $(s_\veps)_\veps$ strongly converges to $0$ in e.g. $L^\infty_T(H^{-\g}_{\rm loc})$ for any~$0<\g<1$.

On the other hand, uniform bounds for $(u_\veps)_\veps$ in $L^2_T(H^1)$ and a paraproduct decomposition immediately imply that
$u_\veps\otimes u_\veps$ is uniformly bounded in $L^1_T(H^{1-\delta})$, for all $\delta>0$. Here, we have used Corollary \ref{c:product} (v) in the Appendix, together with the continuous embedding
$H^1\hra B^{0}_{\infty,\infty}$.

Therefore, by a direct application of Corollary \ref{c:product} (i) in the Appendix, we gather that the quantity $\bigl(\rho_\veps u_\veps\otimes u_\veps\bigr)_\veps$ weakly converges to $0$
(up to a suitable extraction of a subsequence) in the space $L^1_T(H^{-\g-\delta})$, for any $\delta>0$ arbitrarily small.
\end{proof}

The next step consists in regularizing the velocity fields in the convection term: thanks to the smoothness of $\rho_0$, we can establish an analogue of Lemma \ref{l:uxu_conv},
whose proof is exactly the same so is omitted.
\begin{lemma} \label{l:uxu_dens}
For any test function $\psi\,\in\,\mc{D}\bigl([0,T[\,\times\Omega\bigr)$, one has
$$
\lim_{M\ra+\infty}\;\limsup_{\veps\ra0^+}\;
\left|\int^T_0\int_\Omega\rho_0\,u_\veps\otimes u_\veps:\nabla\psi\,dx\,dt\,-\,\int^T_0\int_\Omega\rho_0\,u_{\veps,M}\otimes u_{\veps,M}:\nabla\psi\,dx\,dt\right|\,=\,0\,.
$$
\end{lemma}
Thanks to the previous statements, we  just have  to study, for any fixed $M\in\N$, the convergence for $\veps\ra0$ of the integral
$$
-\int^T_0\int_\Omega\rho_0\,u_{\veps,M}\otimes u_{\veps,M}:\nabla\psi\,dx\,dt\,=\,\int^T_0\int_\Omega\div\bigl(\rho_0\,u_{\veps,M}\otimes u_{\veps,M}\bigr)\cdot\psi\,dx\,dt\,,
$$
where, as previously, we have integrated by parts since now we have enough smoothness in the space variable.
We proceed in several steps.

\paragraph*{Step 1.}
By direct differentiation, we  obtain (because $\div u_{\veps,M}=0$)
\begin{eqnarray} 
\div\bigl(\rho_0\,u_{\veps,M}\otimes u_{\veps,M}\bigr) & = & \rho_0\,u_{\veps,M}\cdot\nabla u_{\veps,M}\,+\,
u_{\veps,M}\cdot\nabla\rho_0\,u_{\veps,M} \label{eq:dens_convect} \\
 & = & \frac{1}{2}\,\rho_0\,\nabla\left|u_{\veps,M}\right|^2\,+\,\rho_0\,\omega_{\veps,M}\,u_{\veps,M}^\perp\,+\,u_{\veps,M}\cdot\nabla\rho_0\,u_{\veps,M}\,. \nonumber
\end{eqnarray}

We remark that the first term in the right-hand side of the last equality gives rise to a contribution of the type $\rho_0\,\nabla\Gamma$ in the limit. Analogously, the same fact
happens for any term of the form $\Lambda_{\veps,M}\,\nabla\rho_0$ where~$\Lambda_{\veps,M}$ is any smooth enough distribution: indeed, if we take a test vector field $\psi$ of zero divergence and
we integrate by parts, we get
$$
\int^T_0\int_\Omega\Lambda_{\veps,M}\,\nabla\rho_0\cdot\psi\,dx\,dt\,=\,\int^T_0\int_\Omega\Lambda_{\veps,M}\;\div\!\left(\rho_0\,\psi\right)\,dx\,dt\,=\,
-\int^T_0\int_\Omega\rho_0\,\nabla\Lambda_{\veps,M}\cdot\psi\,dx\,dt\,.
$$
For notational convenience, from now on we will generically denote by $\Gamma_{\veps,M}$ any vector field which can be written under the form
\begin{equation} \label{eq:Gamma}
\Gamma_{\veps,M}\,=\,\rho_0\,\nabla\Lambda^{(1)}_{\veps,M}\,+\,\Lambda^{(2)}_{\veps,M}\,\nabla\rho_0\,,
\end{equation}
with
\begin{equation} \label{cv:Gamma}
\lim_{M\ra+\infty}\,\lim_{\veps\ra0^+}\,\int^T_0\int_\Omega\Gamma_{\veps,M}\,\cdot\,\psi\,dx\,dt\,=\,\lan\rho_0\,\nabla\Gamma\,,\,\psi\ran_{\mc D'\times\mc D}
\end{equation}
for all $\psi\,\in\,\mc D\bigl([0,T[\,\times\Omega\bigr)$ such that $\div\psi=0$, and where $\Gamma$ is a   distribution over $[0,T[\,\times\Omega$.

On the other hand, we will generically denote by $\mc{R}_{\veps,M}$ any remainder, i.e. any term satisfying the property
\begin{equation} \label{eq:remainder}
\lim_{M\ra+\infty}\,\limsup_{\veps\ra0^+}\,\left|\int^T_0\int_\Omega \mc{R}_{\veps,M}\,\cdot\,\psi\,dx\,dt\right|\,=\,0
\end{equation}
for all divergence free test functions $\psi\,\in\,\mc{D}\bigl([0,T[\,\times\Omega\bigr)$.

With these notations, relation \eqref{eq:dens_convect} can be written under the form
$$
\div\bigl(\rho_0\,u_{\veps,M}\otimes u_{\veps,M}\bigr)\,=\,\Gamma_{\veps,M}\,+\,\rho_0\,\omega_{\veps,M}\,u_{\veps,M}^\perp\,+\,u_{\veps,M}\cdot\nabla\rho_0\,u_{\veps,M}\,.
$$

\paragraph*{Step 2.} Let us focus on the vorticity term: by use of Corollary \ref{c:rot_u-V} and especially of the second decomposition for $\eta_{\veps,M}$, we can write
$$
\rho_0\,\omega_{\veps,M}\,u_{\veps,M}^\perp\,=\,\eta_{\veps,M}\,u_{\veps,M}^\perp\,-\,u_{\veps,M}\cdot\nabla^\perp\rho_0\,u_{\veps,M}^\perp\,-\,
\veps^\theta\,\curl\z_{\veps,M}\,u_{\veps,M}^\perp\,-\,\curl h_{\veps,M}\,u_{\veps,M}^\perp\,,
$$
where, in view of Proposition \ref{p:u-V} and of estimates \eqref{ub_dens:u_2}, the last two terms contribute as remainders, in the sense of relation \eqref{eq:remainder}.
%% Indeed, notice that, in \eqref{ub_dens:u_2}, by uniform bounds we get that the constant $C(k,M)$ does not depend anymore on $M$ if we take $k=0$.
Plugging this expression into the equation for the convective term, we arrive at
\begin{equation} \label{eq:dens_convect_2}
\div\bigl(\rho_0\,u_{\veps,M}\otimes u_{\veps,M}\bigr)\,=\,\Gamma_{\veps,M}\,+\,\mc{R}_{\veps,M}\,+\,\eta_{\veps,M}\,u^\perp_{\veps,M}\,-\,
u_{\veps,M}\cdot\nabla^\perp\rho_0\,u_{\veps,M}^\perp\,+\,u_{\veps,M}\cdot\nabla\rho_0\,u_{\veps,M}\,.
\end{equation}

\paragraph*{Step 3.}
Let us concentrate now on the term
\begin{equation} \label{def:X}
X_{\veps,M}\,:=\,-\,\left(u_{\veps,M}\cdot\nabla^\perp\rho_0\right)\,u_{\veps,M}^\perp\,+\,\left(u_{\veps,M}\cdot\nabla\rho_0\right)\,u_{\veps,M}\,.
\end{equation}
For reasons which will appear clear in a while (see also~\cite{G-SR_2006}), let us introduce a function $b\in\mc{C}^{\infty}_0(\R^2)$, with $0\leq b(y)\leq1$, such that
$b\equiv1$ on $\left\{|y|\leq1\right\}$ and $b\equiv0$ on $\left\{|y|\geq2\right\}$. For any $x\in\Omega$, we define
\begin{equation} \label{eq:b_M}
b_M(x)\,:=\,b\bigl(2^{M/2}\,\nabla\rho_0(x)\bigr)\,.
\end{equation}
To begin with, we notice that, once a $2<q<+\infty$ is fixed, for any compact set $K\subset\Omega$ we can estimate
$$
\left\|b_M\,X_{\veps,M}\right\|_{L^1([0,T]\times K)}\,\leq\,C\,\left\|X_{\veps,M}\right\|_{L^1_T(L^{q/2})}\,
\left(\mu\left\{x\in\R^2\;\bigl|\;\left|\nabla\rho_0(x)\right|\,\leq\,2^{1-M/2}\right\}\right)^{\!(q-2)/q}\,.
$$
Since we can bound $\left\|X_{\veps,M}\right\|_{L^1_T(L^{q/2})}$ by the quantity $C\,\left\|\nabla\rho_0\right\|_{L^\infty}\,\left\|u_{\veps,M}\right\|^2_{L^2_T(L^q)}$,
which is controlled by a uniform constant in view of \eqref{ub:u_L^p} and Sobolev embeddings,
assumption \eqref{eq:non-crit} tells us that the term $b_{M}\,X_{\veps,M}$ is a remainder in the sense of relation \eqref{eq:remainder}.

On the other hand, on the support of $1-b_M$, the vector $\nabla\rho_0$ is far from $0$, so that we can decompose $u_{\veps,M}$ along the basis of $\R^2$
(conveniently renormalized) $\left\{\nabla\rho_0\,,\,\nabla^\perp\rho_0\right\}$.
More precisely, we can write
\begin{eqnarray*}
(1-b_M)\,u_{\veps,M}\,=\,(1-b_M)\,\frac{1}{|\nabla\rho_0|^2}\left(\bigl(u_{\veps,M}\cdot\nabla\rho_0\bigr)\,\nabla\rho_0\,+\,
\bigl(u_{\veps,M}\cdot\nabla^\perp\rho_0\bigr)\,\nabla^\perp\rho_0\right) \\[1ex]
(1-b_M)\,u^\perp_{\veps,M}\,=\,(1-b_M)\,\frac{1}{|\nabla\rho_0|^2}\left(\bigl(u^\perp_{\veps,M}\cdot\nabla\rho_0\bigr)\,\nabla\rho_0\,+\,
\bigl(u_{\veps,M}\cdot\nabla\rho_0\bigr)\,\nabla^\perp\rho_0\right)\,.
\end{eqnarray*}

Notice that, in view of \eqref{eq:Gamma}, we can write the terms in $\nabla\rho_0$ in the generic form $\Gamma_{\veps,M}$.
Therefore, in the end, by definition \eqref{def:X} of $X_{\veps,M}$ we get
\begin{eqnarray*}
(1-b_M)\,X_{\veps,M} & = & \Gamma_{\veps,M}\,-\,(1-b_M)\,\frac{1}{|\nabla\rho_0|^2}\,\bigl(u_{\veps,M}\cdot\nabla^\perp\rho_0\bigr)\,\bigl(u_{\veps,M}\cdot\nabla\rho_0\bigr)\,\nabla^\perp\rho_0\,+ \\
& & \qquad+\,(1-b_M)\,\frac{1}{|\nabla\rho_0|^2}\,\bigl(u_{\veps,M}\cdot\nabla\rho_0\bigr)\,
\bigl(u_{\veps,M}\cdot\nabla^\perp\rho_0\bigr)\,\nabla^\perp\rho_0 \\
& = & \Gamma_{\veps,M}\,,
\end{eqnarray*}
due to the very special cancellations of the two terms in $\nabla^\perp\rho_0$.

Finally, we obtain that $X_{\veps,M}$ contributes as a term of the form \eqref{eq:Gamma}. Hence, in view of \eqref{eq:dens_convect_2}, we have
$$
\div\bigl(\rho_0\,u_{\veps,M}\otimes u_{\veps,M}\bigr)\,=\,\Gamma_{\veps,M}\,+\,\mc{R}_{\veps,M}\,+\,\eta_{\veps,M}\,u^\perp_{\veps,M}\,.
$$

\paragraph*{Step 4.}
In order to treat the last term in the right-hand side of the previous equation, we resort again to the function $b_M$ defined in \eqref{eq:b_M} above.
We also exploit the first decomposition of $\eta_{\veps,M}$ given in Corollary \ref{c:rot_u-V}.

Thanks to the properties stated in Corollary \ref{c:rot_u-V}, for any $T>0$ and any compact set $K\subset\Omega$, it is an easy matter to estimate
\begin{eqnarray*}
& & \hspace{-0.5cm}
\left\|b_M\,\eta_{\veps,M}\,u^\perp_{\veps,M}\right\|_{L^1([0,T]\times K)}\;\leq\;C\,\veps^\theta\,\left\|\eta^{(2)}_{\veps,M}\right\|_{L^2_T(L^2)}\,\left\|u_{\veps,M}\right\|_{L^2_T(L^2)}\,+ \\
& & \qquad\qquad+\,C\,\left\|\eta^{(1)}_{\veps,M}\right\|_{L^2_T(L^2)}\,\left\|u_{\veps,M}\right\|_{L^2_T(L^q)}\,
\left(\mu\left\{x\in\R^2\;\bigl|\;\left|\nabla_h\rho_0(x)\right|\,\leq\,2^{1-M/2}\right\}\right)^{\!1/m} \\
& & \qquad\qquad\qquad\qquad\qquad\leq\;C(M,T)\,\veps^\theta\,+\,C\,\left(\mu\left\{x\in\R^2\;\bigl|\;\left|\nabla_h\rho_0(x)\right|\,\leq\,2^{1-M/2}\right\}\right)^{\!1/m}\,,
\end{eqnarray*}
where $q$ and $m$, both larger than $2$, are linked by the relation $1/q\,+\,1/m\,=\,1/2$. Therefore, in view of assumption \eqref{eq:non-crit}, also this term is a remainder, in the
sense that it gives a contribution of the form \eqref{eq:remainder}.

For the other term $(1-b_M)\,\eta_{\veps,M}\,u^\perp_{\veps,M}$, we decompose the velocity field along $\left\{\nabla\rho_0\,,\,\nabla^\perp\rho_0\right\}$, as done before.
Forgetting about the term along $\nabla\rho_0$, since is a contribution of the form $\Gamma_{\veps,M}$ (recall \eqref{eq:Gamma} above), we can write
\begin{eqnarray*}
(1-b_M)\,\eta_{\veps,M}\,u^\perp_{\veps,M} & = & \Gamma_{\veps,M}\,+\,
(1-b_M)\,\frac{1}{|\nabla\rho_0|^2}\,\eta_{\veps,M}\,\bigl(u_{\veps,M}\cdot\nabla\rho_0\bigr)\,\nabla^\perp\rho_0 \\
& = & \Gamma_{\veps,M}\,+\,\frac{1-b_M}{|\nabla\rho_0|^2}\,\eta_{\veps,M}\left(\div V_{\veps,M}-\veps^\theta\,\div\z_{\veps,M}-\div h_{\veps,M}\right)\nabla^\perp\rho_0\,,
\end{eqnarray*}
where we have used also Corollary \ref{c:rot_u-V}. Notice that the terms presenting $\veps^\theta\,\div\z_{\veps,M}$ and $\div h_{\veps,M}$ give rise to remainders, in the sense of
\eqref{eq:remainder}. For the former item, this fact is obvious thanks to the presence of $\veps^\theta$. As for the latter, it is enough to decompose again
$\eta_{\veps,M}=\eta^{(1)}_{\veps,M}+\veps^\theta\eta^{(2)}_{\veps,M}$ and to use the bounds for $h_{\veps,M}$ given in Proposition \ref{p:u-V}.

Putting all these facts together, we discover that the convective term can be written as
$$
\div\bigl(\rho_0\,u_{\veps,M}\otimes u_{\veps,M}\bigr)\,=\,\Gamma_{\veps,M}\,+\,\mc{R}_{\veps,M}\,+\,
(1-b_M)\,\frac{1}{|\nabla\rho_0|^2}\,\eta_{\veps,M}\,\div V_{\veps,M}\,\nabla^\perp\rho_0\,.
$$

\paragraph*{Step 5.} As final step, we use the first relation in \eqref{reg:approx-w} to write
\begin{eqnarray*}
\eta_{\veps,M}\,\div V_{\veps,M} & = & -\,\veps\,\eta_{\veps,M}\,\d_t\s_{\veps,M}\,=\,-\,\veps\,\bigl(\eta_{\veps,M}-\s_{\veps,M}\bigr)\,\d_t\s_{\veps,M}\,-\,
\frac{\veps}{2}\,\frac{d}{dt}|\s_{\veps,M}|^2 \\
& = & -\,\veps\,\frac{d}{dt}\Bigl(\bigl(\eta_{\veps,M}-\s_{\veps,M}\bigr)\,\s_{\veps,M}\,+\,|\s_{\veps,M}|^2/2\Bigr)\,+\,\veps\,\curl f_{\veps,M}\,\s_{\veps,M}\,,
\end{eqnarray*}
where we also exploited the difference of the two equations in \eqref{reg:approx-w}, recall \eqref{eq:vort_form}.

In the end, in light also of \eqref{reg:source} and uniform bounds, the previous relations prove that
$$
(1-b_M)\,\frac{1}{|\nabla\rho_0|^2}\,\eta_{\veps,M}\,\div V_{\veps,M}\,\nabla^\perp\rho_0\,=\,\mc{R}_{\veps,M}\,,
$$
which finally implies that
\begin{equation} \label{eq:convect_fin}
\div\bigl(\rho_0\,u_{\veps,M}\otimes u_{\veps,M}\bigr)\,=\,\Gamma_{\veps,M}\,+\,\mc{R}_{\veps,M}\,.
\end{equation}

\begin{rem} \label{r:comp-comp}
We notice here that, \emph{in absence of vacuum}, we could have alternatively worked on a different approximation of the convective term, namely on
$$
\div\left(\frac{1}{\rho_0}\,V_{\veps,M}\otimes V_{\veps,M}\right)\,.
$$
An analogous compensated compactness argument (see also \cite{F-G-GV-N}) would have led us to the same conclusion: namely, the transport term vanishes in the limit, up to a remainder
of the form $\rho_0\,\nabla\Gamma$,  due to the strong constraint which is imposed by having a variable target density.

The advantage of our approach is that here we do not need the bound $\rho_0\geq\rho_*>0$.
\end{rem}

\begin{rem} \label{r:Gamma}
Notice that the terms contributing to $\Gamma_{\veps,M}$ are either quadratic in $u_{\veps,M}$ (see Steps~1 and 3), or they depend on the product
$\eta_{\veps,M}\,u_{\veps,M}$ (see Step 4). In view of Corollary \ref{c:rot_u-V}, we conclude that these terms are uniformly bounded,
both with respect to $\veps$ and $M$, in $L^1_T(L^1)$, up to a remainder of order $O(\veps^\theta)$. So
we can assume that the distribution $\Gamma$ in \eqref{cv:Gamma} is a Radon measure over $[0,+\infty[\,\times\Omega$ and that the convergence holds true
in the weak-$*$ topology of the space~$\mc M([0,T]\times\Omega)$ of Radon measures over $[0,T]\times\Omega$ (see Comments to Chapter 4 in \cite{Brezis}).
\end{rem}

\subsection{Conclusion: passing to the limit} \label{ss:limit}

We are now in position of passing to the limit in the modified weak formulation \eqref{eq:vort_weak}, completing in this way the proof to Theorem \ref{t:dens}.

We have already seen that we can pass to the limit with no difficulty in the terms $V_\veps$, $\s_{\veps}$ and in the viscosity term, as well as in the initial data.
On the other hand, putting together Lemmas~\ref{l:dens_conv} and \ref{l:uxu_dens}, %%Steps 1 to 5 of Paragraph \ref{sss:dens-cc}
relation \eqref{eq:convect_fin} and Remark \ref{r:Gamma}, we gather
the convergence
\begin{equation} \label{eq:conv_convect}
\int^T_0\int_\Omega\rho_\veps\,u_\veps\otimes u_\veps:\nabla\psi\,dx\,dt\;\longrightarrow\;
-\,\lan\Gamma\,,\,\div(\rho_0\,\psi)\ran_{\mc M\times\mc C}\,=\,\lan\rho_0\,\nabla\Gamma\,,\,\psi\ran_{\mc D'\times\mc D}
\end{equation}
when $\veps\ra0$, for all test functions $\psi\,\in\,\mc{D}\bigl([0,T[\,\times\Omega\bigr)$ such that $\div\psi=0$.
%where, by a little abuse of notation, we have identified
%$$
%\int^T_0\int_\Omega\rho_0\,\Gamma\,:\,\nabla\psi\,dx\,dt\,:=\,\lan\Gamma\,,\,\rho_0\,\nabla\psi\ran_{\mc M\times\mc C}\,=\,
%\lan\rho_0\,\nabla\Gamma\,,\,\psi\ran_{\mc M\times\mc C}\,.
%$$
In the previous formula, the notation $\lan\cdot,\cdot\ran_{\mc M\times\mc C}$ denotes the duality product between Radon measures and continuous functions.

Therefore, starting from \eqref{eq:vort_weak}, where we have set $\psi\,=\,\nabla^\perp\phi$, in the limit for $\veps\ra0$  we find
\begin{align*}
&-\int^T_0\int_\Omega \rho_0\,u\cdot\d_t\nabla^\perp\phi\,-\,\int^T_0\int_\Omega\s\,\d_t\phi\,- \\
&\qquad-\,\int^T_0\int_\Omega\Gamma\,\div\bigl(\rho_0\,\nabla^\perp\phi\bigr)\,+\,\nu\int^T_0\int_\Omega\nabla u:\nabla\nabla^\perp\phi\,=\,
\int_\Omega m_0\cdot\nabla^\perp\phi(0)\,+\,\int_\Omega r_{0}\,\phi(0)\,,
\end{align*}
where $\phi\,\in\mc D\bigl([0,T[\,\times\Omega\bigr)$ is any test function, with no additional restrictions, and where with a little abuse of notation, we have identified
$$
-\,\int^T_0\int_\Omega\Gamma\,\div\bigl(\rho_0\,\nabla^\perp\phi\bigr)\,:=\,-\,\lan\Gamma\,,\,\div\bigl(\rho_0\,\nabla^\perp\phi\bigr)\ran_{\mc M\times\mc C}\,.
$$
A further integration by parts allows us to arrive (after multiplying by $-1$) at the equation
\begin{align*}
&-\int^T_0\!\!\int_\Omega\Bigl(\curl\bigl(\rho_0\,u\bigr)-\s\Bigr)\cdot\d_t\phi\,dx\,dt\,+ \\
&\qquad+\,\int^T_0\int_\Omega\Gamma\,\div\bigl(\rho_0\,\nabla^\perp\phi\bigr)\,dx\,dt\,+\,\nu\int^T_0\!\!\int_\Omega\nabla\omega:\nabla\phi\,dx\,dt\,=\,
\int_\Omega\Bigl(\curl m_0\,-\,r_0\Bigr)\phi(0)\,dx\,,
\end{align*}
which coincides with the weak formulation of equation \eqref{eq:dens_lim}, in view of \eqref{eq:conv_convect}.
%%Indeed, it enough to notice that, integrating formally by parts twice, one has
%%$\int^T_0\!\!\int_\Omega\Gamma\,\div\bigl(\rho_0\,\nabla^\perp\phi\bigr)\,=\,\int^T_0\!\!\int_\Omega\curl\bigl(\rho_0\,\nabla\Gamma\bigr)\,\phi$.

\medskip 
Theorem \ref{t:dens} is thus completely proved.\qed

%%%%%%%%%%%%%%%%%%%%%%%%%%%%%%%%%%%%%%%%%%%%%%%%%%%%%%%%%%%%
\subsection{A conditional convergence result} \label{ss:dens-full}
%%%%%%%%%%%%%%%%%%%%%%%%%%%%%%%%%%%%%%%%%%%%%%%%%%%%%%%%%%%%

In this subsection, we state and prove a convergence result for the fully non-homogeneous case, where we are able to pass to the limit to the full system,
in which the dynamics of the density fluctuation function and the velocity field are decoupled.

This is just a conditional result, because  very strong assumptions are required on the family of weak solutions: in particular, we need to assume uniform bounds in higher norms for the family
of velocity fields (see in particular conditions (ii) and (iii) in Theorem \ref{t:dens-full} below), which cannot be deduced from classical energy estimates.

The statement is the following.
\begin{thm} \label{t:dens-full}
With the notation and under the assumptions of Theorem {\rm\ref{t:dens}}, assume moreover that $\rho_0\,\in\,W^{3,\infty}(\Omega)$ and that the following conditions hold true:
\begin{enumerate}[(i)]
 \item $\bigl(r_{0,\veps}\bigr)_\veps\,\subset\,H^{1+\beta}(\Omega)$, for some $\beta\in\,]0,1[\,$;
 %%% and $\bigl(m_{0,\veps}\bigr)_\veps\,\subset\,H^{\beta_1}(\Omega)$, for some $\beta_0$ and $\beta_1$ in $\,]0,1[\,$;
 \item $\bigl(u_\veps\bigr)_\veps\,\subset\,L^{\infty}_{\rm loc}\bigl(\R_+;H^{1}(\Omega)\bigr)\cap L^{2}_{\rm loc}\bigl(\R_+;H^{2}(\Omega)\bigr)$;
 \item $\bigl(u_\veps\bigr)_\veps\,\subset\,\mc C^{0,\g}_{\rm loc}\bigl(\R_+;L^{2}(\Omega)\bigr)$, for some $\g\in\,]0,1[\,$.
\end{enumerate}

Then there exist distributions $\Pi$, $\Gamma_0$ and $\Gamma_1$ over $\R_+\times\Omega$ such that the limit points $\sigma$ and $u$ satisfy the system
\begin{equation} \label{eq:dens-full}
\left\{\begin{array}{l}
        \d_t\s\,+\,u\cdot\nabla\s\,=\,\curl\bigl(\rho_0\,\nabla\Gamma_1\bigr) \\[1ex]
        \rho_0\,\d_tu\,+\,\nabla\Pi\,+\,\rho_0\,\nabla\Gamma_0\,+\,\s\,u^\perp\,-\,\nu\,\Delta u\,=\,0 \\[1ex]
         \div u\,=\,\div\bigl(\rho_0\,u\bigr)\,=\,0\,,
       \end{array}
\right. 
\end{equation}
with initial data respectively $\s_{|t=0}\,=\,r_0$ and $\bigl(\rho_0\,u\bigr)_{|t=0}\,=\,m_0$, where $r_0$ and $m_0$ have been defined in \eqref{eq:conv-initial}.
\end{thm}

It goes without saying that, under conditions (i)-(ii)-(iii), the convergence properties for $\bigl(u_\veps\bigr)_\veps$ and $\bigl(\s_\veps\bigr)_\veps$
stated in Theorem \ref{t:dens-full} can be improved (see also Proposition \ref{p:full-sigma_e} below).
However, our focus here is on obtaining convergence to the full system rather than~(\ref{eq:dens_lim}). In this respect, let us make some comments.

\begin{rem} \label{r:dens-full}
In order to make sense of  the equation on $\sigma_\veps$ and to pass to the limit, we need enough regularity on $\sigma_\veps$ in order to define  the product $\sigma_\veps\,u_\veps$, and to prove uniform
bounds.
Now, the only way to get information for $\sigma_\veps$ is to exploit equation \eqref{eq:vort_form}: assumptions (i) and (ii) of Theorem \ref{t:dens-full}
are designed so as  to get the desired smoothness for all terms appearing in that equation, while assumption (iii) is needed in order to
prove compactness properties for $\bigl(\sigma_\veps\bigr)_\veps$.
\end{rem}

\begin{rem} \label{r:full-assump}
It is well-known that, for $2$-D density-dependent Navier-Stokes equations \emph{in the absence of vacuum}, if the initial velocity field is in $H^1$, then there exists
a weak solution $u\,\in\,\mc C_T(H^1)\cap L^2_T(H^2)$. We refer e.g. to pages 31-32 of \cite{Lions_1} and to comments to Theorem 3.41 in \cite{B-C-D} for precise statements and
further details. Moreover, an $H^{1+\beta}$ assumption on the initial density is required for recovering uniqueness; see e.g. Section 5 of \cite{D_2004}.

Nonetheless, it seems to us not possible to get \emph{uniform bounds} for $\bigl(u_\veps\bigr)_\veps$ in $L^\infty_T(H^1)\cap L^2_T(H^2)$, even in absence of vacuum, using the strategy of proof adopted
in the above mentioned results. The problem is that the Coriolis term, which does not affect basic energy estimates, does affect these higher order type estimates. This is why we have
to assume the \textsl{a priori} bounds  (ii).
\end{rem}

\begin{rem} \label{r:Lagrangian}
As it will be apparent from the proof, see in particular Paragraph \ref{sss:full_conv} below, the regularities of the distributions $\Gamma_0$ and $\Gamma_1$ with respect
to the space variable will be enough to make sense of the product of their gradiens by $\rho_0$ in $\mc D'$.
\end{rem}

The proof of Theorem \ref{t:dens-full} consists in two main steps. In the first one, in Paragraph \ref{sss:full_dens-var}, we  study   the density variations: the new assumptions allow to improve the regularity of the $\sigma_\veps$'s. This turns out to be fundamental in
passing to the limit  $\veps\ra0$, as     shown in Paragraph \ref{sss:full_conv}

\subsubsection{Study of the density variations} \label{sss:full_dens-var}

First of all, let us consider the functions $s_\veps\,:=\,\rho_\veps\,-\,\rho_0$, which satisfy equation \eqref{eq:s_e}. We can establish the following result.
\begin{lemma} \label{l:full-s_e}
Under assumptions {\rm(i)} and {\rm(ii)} of Theorem {\rm\ref{t:dens-full}}, for all $0<\alpha<\beta$, one has the uniform bound
$$
\bigl(s_\veps\bigr)_\veps\;\subset\;L^\infty_{\rm loc}\bigl(\R_+;H^{1+\alpha}(\Omega)\bigr)\,\cap\,\mc C^{0,1/2}_{\rm loc}\bigl(\R_+;H^\alpha(\Omega)\bigr)\,.
$$
In particular, this family is compact e.g. in the space $L^\infty_T(H_{\rm loc}^{\alpha})$ for all fixed times $T>0$. 
\end{lemma}

\begin{proof}
The starting point is the transport equation \eqref{eq:s_e}.
By assumption, the family of initial data~$\bigl(\veps\,r_{0,\veps}\bigr)_\veps$ is uniformly bounded in $H^{1+\beta}$,
and the family of external forces $\bigl(u_\veps\cdot\nabla\rho_0\bigr)_\veps$ is uniformly bounded in $L^2_T(H^2)\,\hookrightarrow\,L^2_T(H^{1+\beta})$.
Moreover, the family of divergence-free vector fields $\bigl(u_\veps\bigr)_\veps$ is uniformly bounded in $L^2_T(H^{2})$.
Therefore, for all $0<\alpha<\beta$, by an application of Proposition 5.2 of \cite{D_2004} we deduce that~$\bigl(s_\veps\bigr)_\veps\,\subset\,L^\infty_T(H^{1+\alpha})$ for all fixed~$T>0$.

On the other hand, we can write $\d_ts_\veps\,=\,-\div\bigl(s_\veps\,u_\veps\bigr)\,-\,u_\veps\cdot\nabla\rho_0$: because $H^{1+\alpha}$ is a Banach algebra,   we infer
the uniform bound $\bigl(\d_ts_\veps\bigr)_\veps\,\subset\,L^2_T(H^\alpha)$ for all times~$T>0$, which implies that~$\bigl(s_\veps\bigr)_\veps\,\subset\,\mc C^{0,1/2}_T(H^\alpha)$.

Finally, the last assertion of the statement follows by the Ascoli-Arzel\`a theorem and interpolation of the previous two properties. The lemma is   proved.
\end{proof}

Thanks to the previous result, we can improve the regularity for the functions $\sigma_\veps$.
\begin{prop} \label{p:full-sigma_e}
Under assumptions {\rm(i)} and {\rm(ii)} of Theorem {\rm\ref{t:dens-full}}, one has the uniform bound
$$
\bigl(\s_\veps\bigr)_\veps\;\subset\;L^\infty_{\rm loc}\bigl(\R_+;H^{-1}(\Omega)\bigr)\,.
$$
In particular, there holds $
\bigl(\s_\veps\,u_\veps\bigr)_\veps\;\subset\;L^2_{\rm loc}\bigl(\R_+;H^{-1}(\Omega)\bigr)
$.

Moreover, under assumption {\rm(iii)} $\bigl(\s_\veps\bigr)_\veps$ is compact in $L^\infty_T(H_{\rm loc}^{-1-k-\delta})$, for any~$0<\delta<1$ and for all times $T>0$,
where $k$ is the index fixed in Proposition~{\rm\ref{p:s-uniform}}. In particular, one gathers also the convergence property
$$
\sigma_\veps\,u_\veps\,\rightharpoonup\,\s\,u\quad \mbox{ in }\quad L^2_T(H_{\rm loc}^{-1-k-\delta})\,.
$$
\end{prop}

\begin{proof}
Integrating equation \eqref{eq:vort_form} in time, for all $\veps\in\,]0,1[\,$ we can write
$$
\sigma_\veps(t,x)\,=\,\veps\,r_{0,\veps}(x)\,+\,\eta_\veps(t,x)\,-\,\curl\bigl(m_{0,\veps}(x)\bigr)\,+\,\int^t_0\curl f_\veps(\tau,x)\,d\tau\,,
$$
where we recall that $\eta_\veps\,:=\,\curl\bigl(\rho_\veps\,u_\veps\bigr)$ and $f_\veps$ is defined in \eqref{def:svf}.

By assumption on the initial data, we have that both $\bigl(r_{0,\veps}\bigr)_\veps$ and $\bigl(\curl m_{0,\veps}\bigr)_\veps$ are uniformly bounded in $H^{-1}$.
Furthermore, writing $\rho_\veps\,=\,\rho_0\,+\,s_\veps$, by Lemma \ref{l:full-s_e} and product rules we deduce that $\bigl(\eta_\veps\bigr)_\veps\,\subset\,L^2_T(L^2)$.

For convenience, let us choose $\alpha=\beta/2$ in Lemma \ref{l:full-s_e}, and keep it fixed throughout   this proof.
We now consider the convection term coming into play in the definition of $f_\veps$.
We remark, by assumption (ii) of Theorem \ref{t:dens-full} and interpolation, that~$\bigl(u_\veps\bigr)_\veps$  is uniformly bounded in~$ L^4_T(H^{3/2})$.
Therefore, the family   $ \rho_0\,u_\veps\otimes u_\veps$ is uniformly bounded in the space~ $L^1_T(H^2)\cap L^2_T(H^{3/2})$, while
the family of $ s_\veps\,u_\veps\otimes u_\veps$ is uniformly bounded in $L^2_T(H^{1+\alpha})$, since we have chosen~$\alpha<1/2$.
On the other hand, $\bigl(\Delta\omega_\veps\bigr)_\veps$  is uniformly bounded in~$L^2_T(H^{-1})$, so we finally deduce the uniform bound~$\bigl(\curl f_\veps\bigr)_\veps\,\subset\,L^2_T(H^{-1})$.

Putting all these properties together,   we deduce that the family $\bigl(\sigma_\veps\bigr)_\veps$ is uniformly bounded in~$L^\infty_T(H^{-1})$, for all $T>0$.
By product rules of Corollary \ref{c:product} (iv), the uniform bounds for~$\bigl(\s_\veps\,u_\veps\bigr)_\veps$ follow immediately.

\medbreak
Remark that, by \eqref{eq:vort_form} again and the above study on $\bigl(f_\veps\bigr)_\veps$, it follows also that $\bigl(\eta_\veps-\sigma_\veps\bigr)_\veps$ is compact
in the space $\mc C^{0,1/2-\delta}_T(H_{\rm loc}^{-1-\delta})$, for all $0<\delta<1$.
On the other hand, thanks to Proposition~\ref{p:s-uniform}, we can write
$$
\eta_\veps\,=\,\curl\bigl(\rho_0\,u_\veps\bigr)\,+\,\veps^\theta\,\curl\left(\veps^{-\theta}\,s_\veps\,u_\veps\right)\,.
$$
Observe that $\bigl(u_\veps\bigr)_\veps\,\subset\,L^\infty_T(H^1)$, so we deduce from Proposition \ref{p:s-uniform} that
$\left(\veps^{-\theta}\,s_\veps\,u_\veps\right)_\veps$ is uniformly bounded in $L^\infty_T(H^{-k-\delta})$, for any $\delta>0$ small
(we could actually improve this property, in view of Lemma \ref{l:full-s_e} above, but it is enough for our purposes).
Therefore, the latter term in the right-hand side converges strongly, as $\veps\ra0$, in the space $L^\infty_T(H^{-k-\delta-1})$ for all $\delta>0$ arbitrarily small. Moreover,
by assumption (iii) we gather that the former term is uniformly bounded in $\mc C^{0,\g}_T(H^{-1})$, and so it is compact (by the Ascoli-Arzel\`a theorem) in
$\mc C^{0,\g-\delta}_T(H_{\rm loc}^{-1-\delta})$, where again $\delta>0$ is arbitrarily small.
 From the previous properties, we immediately deduce that $\bigl(\sigma_\veps\bigr)_\veps$ is compact in the space~$L^\infty_T(H_{\rm loc}^{-1-k-\delta})$, for any $T>0$ fixed and any $\delta>0$ small.

In particular, we can assume, with no loss of generality, that $1+k+\delta<2$. Hence, combining this property with product rules of Corollary \ref{c:product} (iv) and
the previous uniform bounds, we discover that
$\bigl(\sigma_\veps\,u_\veps\bigr)_\veps$ is weakly convergent to the product $\sigma\,u$ in $L^2_T(H_{\rm loc}^{-1-k-\delta})$.
\end{proof}

\subsubsection{Passing to the limit} \label{sss:full_conv}

In view of the previous results, it is possible to pass to the limit in system \eqref{eq:dd-NSC}.

First of all, we notice that we are now allowed to write (still in the weak sense) the equation
\begin{equation} \label{eq:full_sigma_e}
\d_t\s_\veps\,+\,\div\bigl(\s_\veps\,u_\veps\bigr)\,=\,-\,\veps^{-1}\,u_\veps\cdot\nabla\rho_0\,.
\end{equation}
As a consequence of this relation and Proposition \ref{p:full-sigma_e}, arguing exactly in the same way as in establishing \eqref{ub:u-r_0}, we find
the uniform embedding property
$$ %%\begin{equation} \label{ub:full_u-r_0}
\left({\veps^{-1}}\,u_\veps\cdot\nabla\rho_0\right)_\veps\;\subset\;W^{-1,\infty}_T(H^{-1})\,+\,L^2_T(H^{-1})\;\hookrightarrow\;
W^{-1,\infty}\bigl([0,T];H^{-1}(\Omega)\bigr)\,,
$$ %%\end{equation}
and then this term has to weakly converge, in the previous space, to some~$g\,\in\,W^{-1,\infty}_T(H^{-1})$.
Observe that, to be consistent with   Theorem \ref{t:dens},  $g$ has to be equal to $\curl\bigl(\rho_0\nabla\Gamma_1\bigr)$, for some
scalar distribution $\Gamma_1$. Indeed, since $u_\veps$ is divergence-free, we can write~$u_\veps\,=\,\nabla^\perp\Psi_\veps$, hence
$$
u_\veps\cdot\nabla\rho_0\,=\,\nabla^\perp\Psi_\veps\cdot\nabla\rho_0\,=\,-\,\nabla\Psi_\veps\cdot\nabla^\perp\rho_0\,=\,-\,\curl\bigl(\rho_0\,\nabla\Psi_\veps\bigr)\,.
$$
Therefore, thanks to these computations and Proposition \ref{p:full-sigma_e}, we can pass to the limit in equation~\eqref{eq:full_sigma_e} and find
$$
\d_t\s\,+\,\div\bigl(\s\,u\bigr)\,=\,\curl\bigl(\rho_0\,\nabla\Gamma_1\bigr)\,,
$$
for a suitable distribution $\Gamma_1$ such that $\curl\bigl(\rho_0\,\nabla\Gamma_1\bigr)\in W^{-1,\infty}_T(H^{-1})$ for all $T>0$ fixed.

\medbreak
Let us consider now the weak formulation of the momentum equation:
$$
\int^T_0\!\!\int_{\Omega}\biggl(-\rho_\veps u_\veps\cdot\d_t\psi\,-\,\rho_\veps u_\veps\otimes u_\veps:\nabla\psi\,+\,\frac{1}{\veps}\,\rho_\veps u_\veps^\perp\cdot\psi\,+\,
\nu\nabla u_\veps:\nabla\psi\biggr)\,dx\,dt\,=\,\int_{\Omega}m_{0,\veps}\cdot\psi(0)\,dx\,,
$$
where $\psi$ is a smooth divergence-free test function, compactly supported in $[0,T[\,\times\Omega$.
Obviously, the convergence of the $\d_t$ term, the viscosity term and the initial datum present no difficulty, and it can be performed as in the previous paragraphs.
Furthermore, the computations made in Subsection \ref{ss:convective} allow us to pass to the limit also in the convective term, as done in \eqref{eq:conv_convect},
for a suitable Radon measure $\Gamma\in\mc M([0,T]\times\Omega)$.
So the only changes concern the convergence in the Coriolis term. Using the decomposition~$\rho_{\veps}\,=\,\rho_0\,+\,\veps\,\sigma_\veps$ and the fact that
$u_\veps\,=\,\nabla^\perp\Psi_\veps$, we can write
$$
\frac{1}{\veps}\int^T_0\!\!\int_{\Omega}\rho_\veps u_\veps^\perp\cdot\psi\,=\,\frac{1}{\veps}\int^T_0\!\!\int_{\Omega}\rho_0 u_\veps^\perp\cdot\psi\,+\,
\int^T_0\!\!\int_{\Omega}\s_\veps u_\veps^\perp\cdot\psi\,=\,-\,\frac{1}{\veps}\int^T_0\!\!\int_{\Omega}\rho_0\,\nabla\Psi_\veps\cdot\psi\,+\,\int^T_0\!\!\int_{\Omega}\s_\veps u_\veps^\perp\cdot\psi\,.
$$
Passing to the limit in the latter term in the right-hand side can be done as in the mass equation   above (using Proposition \ref{p:full-sigma_e}); moreover the former term is exactly
the same term (using the fact that~$\psi\,=\,\nabla^\perp\phi$, for some smooth $\phi$) we have encountered in the right-hand side of the mass equation. In the end, we deduce that
$$
\frac{1}{\veps}\int^T_0\!\!\int_{\Omega}\rho_\veps u_\veps^\perp\cdot\psi\,dx\,dt\,\longrightarrow\,\lan\rho_0\nabla\Gamma_1\,,\,\psi\ran_{\mc D'\times\mc D}\,+\,
\int^T_0\!\!\int_{\Omega}\s\,u^\perp\cdot\psi\,dx\,dt\,.
$$

Therefore, after setting $\Gamma_0\,=\,\Gamma\,+\,\Gamma_1$, we have proved that the momentum equation converges, in the weak sense, to the equation
$$
\rho_0\,\d_tu\,+\,\nabla\Pi\,+\,\rho_0\,\nabla\Gamma_0\,+\,\s\,u^\perp\,-\,\nu\,\Delta u\,=\,0\,,
$$
supplemented with the constraints $\div u\,=\,\div\bigl(\rho_0\,u\bigr)\,=\,0$.

\medskip 
 Theorem \ref{t:dens-full} is completely proved. \qed

%%%%%%%%%%%%%%%%%%%%%%%%%%%%%%%%%%%%%%%%%%%%%%%%%%%%%%%%%%%%%%%%%%%%%%%%%%%%%%%%%%%%%%%%%%%%%%%%%%%%%%%%%%
%%%%%%%%%%%%%%%%%%%%%%%%%%%%%%%%%%%%%%%%%%%%%%%%%%%%%%%%%%%%%%%%%%%%%%%%%%%%%%%%%%%%%%%%%%%%%%%%%%%%%%
\appendix
%%%%%%%%%%%%%%%%%%%%%%%%%%%%%%%%%%%%%%%%%%%%%%%%%%%%%%%%%%%%%%%%%%%%%%%%%%%%%%%%%%%%%%%%%%%%%%%%%%%%%%%%%%%%
%%%%%%%%%%%%%%%%%%%%%%%%%%%%%%%%%%%%%%%%%%%%%%%%%%%%%%%%%%%%%%%%%%%%%%%%%%%%%%%%%%%%%%%%%%%%%%%%%%%5

\section{Appendix -- Fourier and harmonic analysis toolbox} \label{app:LP}

We recall here the main ideas of Littlewood-Paley theory, which we exploited in the previous analysis.
We refer e.g. to Chapter 2 of \cite{B-C-D} for details.
For simplicity of exposition, let us deal with the $\R^d$ case; however, the whole construction can be adapted also to the $d$-dimensional torus $\T^d$.

\medbreak
First of all, let us introduce the so called ``Littlewood-Paley decomposition'', based on a non-homogeneous dyadic partition of unity with
respect to the Fourier variable. 
We fix a smooth radial function $\chi$ supported in the ball $B(0,2)$, equal to $1$ in a neighborhood of $B(0,1)$
and such that $r\mapsto\chi(r\,e)$ is nonincreasing over $\R_+$ for all unitary vectors $e\in\R^d$. Set
$\varphi\left(\xi\right)=\chi\left(\xi\right)-\chi\left(2\xi\right)$ and
$\vphi_j(\xi):=\vphi(2^{-j}\xi)$ for all $j\geq0$.

The dyadic blocks $(\Delta_j)_{j\in\Z}$ are defined by\footnote{Throughout we agree  that  $f(D)$ stands for 
the pseudo-differential operator $u\mapsto\mc{F}^{-1}(f\,\mc{F}u)$.} 
$$
\Delta_j\,:=\,0\quad\mbox{ if }\; j\leq-2,\qquad\Delta_{-1}\,:=\,\chi(D)\qquad\mbox{ and }\qquad
\Delta_j\,:=\,\varphi(2^{-j}D)\quad \mbox{ if }\;  j\geq0\,.
$$
We  also introduce the following low frequency cut-off operator:
\begin{equation} \label{eq:S_j}
S_ju\,:=\,\chi(2^{-j}D)\,=\,\sum_{k\leq j-1}\Delta_{k}\qquad\mbox{ for }\qquad j\geq0\,.
\end{equation}
The following classical property holds true: for any $u\in\mc{S}'$, then one has the equality~$u=\sum_{j}\Delta_ju$ in the sense of $\mc{S}'$.
Let us also mention the so-called \emph{Bernstein inequalities}, which explain the way derivatives act on spectrally localized functions.
  \begin{lemma} \label{l:bern}
Let  $0<r<R$.   A constant $C$ exists so that, for any nonnegative integer $k$, any couple $(p,q)$ 
in $[1,+\infty]^2$, with  $p\leq q$,  and any function $u\in L^p$,  we  have, for all $\lambda>0$,
$$
\displaylines{
{\rm supp}\, \widehat u \subset   B(0,\lambda R)\quad
\Longrightarrow\quad
\|\nabla^k u\|_{L^q}\, \leq\,
 C^{k+1}\,\lambda^{k+d\left(\frac{1}{p}-\frac{1}{q}\right)}\,\|u\|_{L^p}\;;\cr
{\rm supp}\, \widehat u \subset \{\xi\in\R^d\,|\, r\lambda\leq|\xi|\leq R\lambda\}
\quad\Longrightarrow\quad C^{-k-1}\,\lambda^k\|u\|_{L^p}\,
\leq\,
\|\nabla^k u\|_{L^p}\,
\leq\,
C^{k+1} \, \lambda^k\|u\|_{L^p}\,.
}$$
\end{lemma}   

By use of Littlewood-Paley decomposition, we can define the class of Besov spaces.
\begin{defin} \label{d:B}
  Let $s\in\R$ and $1\leq p,r\leq+\infty$. The \emph{non-homogeneous Besov space}
$B^{s}_{p,r}$ is defined as the subset of tempered distributions $u$ for which
$$
\|u\|_{B^{s}_{p,r}}\,:=\,
\left\|\left(2^{js}\,\|\Delta_ju\|_{L^p}\right)_{j\geq -1}\right\|_{\ell^r}\,<\,+\infty\,.
$$
\end{defin}
Besov spaces are interpolation spaces between Sobolev spaces. In fact, for any $k\in\N$ and~$p\in[1,+\infty]$
we have the following chain of continuous embeddings:
$$
 B^k_{p,1}\hookrightarrow W^{k,p}\hookrightarrow B^k_{p,\infty}\,,
$$
where  $W^{k,p}$ denotes the classical Sobolev space of $L^p$ functions with all the derivatives up to the order $k$ in $L^p$.
Moreover, for all $s\in\R$ we have the equivalence $B^s_{2,2}\equiv H^s$, with
\begin{equation} \label{eq:LP-Sob}
\|f\|_{H^s}\,\sim\,\left(\sum_{j\geq-1}2^{2 j s}\,\|\Delta_jf\|^2_{L^2}\right)^{1/2}\,.
\end{equation}

As an immediate consequence of the first Bernstein inequality, one gets the following embedding result.
\begin{prop}\label{p:embed}
The space $B^{s_1}_{p_1,r_1}$ is continuously embedded in the space $B^{s_2}_{p_2,r_2}$ for all indices satisfying $p_1\,\leq\,p_2$ and
$$
s_2\,<\,s_1-d\left(\frac{1}{p_1}-\frac{1}{p_2}\right)\qquad\mbox{ or }\qquad
s_2\,=\,s_1-d\left(\frac{1}{p_1}-\frac{1}{p_2}\right)\;\;\mbox{ and }\;\;r_1\,\leq\,r_2\,. 
$$
\end{prop}

We recall also Lemma 2.73 of \cite{B-C-D}.
\begin{lemma} \label{l:Id-S}
If $1\leq r<+\infty$, for any $f\in B^s_{p,r}$ one has
$$
\lim_{j\ra+\infty}\left\|f\,-\,S_jf\right\|_{B^s_{p,r}}\,=\,0\,.
$$
\end{lemma}

Let us now introduce the paraproduct operator (after J.-M. Bony, see \cite{Bony}). Constructing the paraproduct operator relies on the observation that, 
formally, any product  of two tempered distributions $u$ and $v,$ may be decomposed into 
\begin{equation}\label{eq:bony}
u\,v\;=\;T_uv\,+\,T_vu\,+\,R(u,v)\,,
\end{equation}
where we have defined
$$
T_uv\,:=\,\sum_jS_{j-1}u\Delta_j v,\qquad\qquad\mbox{ and }\qquad\qquad
R(u,v)\,:=\,\sum_j\sum_{|j'-j|\leq1}\Delta_j u\,\Delta_{j'}v\,.
$$
The above operator $T$ is called ``paraproduct'' whereas
$R$ is called ``remainder''.
The paraproduct and remainder operators have many nice continuity properties. 
The following ones have been of constant use in this paper (see the proof in e.g. Chapter 2 of \cite{B-C-D}).
\begin{prop}\label{p:op}
For any $(s,p,r)\in\R\times[1,\infty]^2$ and $t>0$, the paraproduct operator 
$T$ maps continuously $L^\infty\times B^s_{p,r}$ in $B^s_{p,r}$ and  $B^{-t}_{\infty,\infty}\times B^s_{p,r}$ in $B^{s-t}_{p,r}$.
Moreover, the following estimates hold:
$$
\|T_uv\|_{B^s_{p,r}}\,\leq\, C\,\|u\|_{L^\infty}\,\|\nabla v\|_{B^{s-1}_{p,r}}\qquad\mbox{ and }\qquad
\|T_uv\|_{B^{s-t}_{p,r}}\,\leq\, C\|u\|_{B^{-t}_{\infty,\infty}}\,\|\nabla v\|_{B^{s-1}_{p,r}}\,.
$$
For any $(s_1,p_1,r_1)$ and $(s_2,p_2,r_2)$ in $\R\times[1,\infty]^2$ such that 
$s_1+s_2>0$, $1/p:=1/p_1+1/p_2\leq1$ and~$1/r:=1/r_1+1/r_2\leq1$,
the remainder operator $R$ maps continuously~$B^{s_1}_{p_1,r_1}\times B^{s_2}_{p_2,r_2}$ into~$B^{s_1+s_2}_{p,r}$.
In the case $s_1+s_2=0$, provided $r=1$, operator $R$ is continuous from $B^{s_1}_{p_1,r_1}\times B^{s_2}_{p_2,r_2}$ with values
in $B^{0}_{p,\infty}$.
\end{prop}

As a corollary of the previous proposition, we deduce the following continuity properties of the product   in Sobolev spaces, which have been used in the course of the analysis.
In the statement, we limit ourselves to the case of space dimension $d=2$,   the only relevant one for this study.
\begin{coroll} \label{c:product}
Let $d=2$.
\begin{itemize}
 \item[(i)] For all $\eta$ and all $\delta$ in $\,]0,1[\,$, such that $1-\eta-\delta>0$, the product is a continuous map from $H^{-\eta}\times H^{1-\delta}$ into $H^{-\eta-\delta}$.
 \item[(ii)] For all $0\leq\eta<1$, the product is a continuous map from $H^{-\eta}\times H^{1}$ into $H^{-\eta-\delta}$ for all $\delta>0$ arbitrarily small.
\item[(iii)] For all $0<\eta\leq1$, the product is a continuous map from $H^{\eta}\times H^{1}$ into $L^{2}$.
\item[(iv)] For all $0<\eta<2$, the product is a continuous map from $H^{-\eta}\times H^{2}$ into $H^{-\eta}$.
\item[(v)] The product is a continuous map from $H^{1}\times H^{1}$ into $H^{1-\delta}$ for all $\delta>0$ arbitrarily small.
\end{itemize}
\end{coroll}

\begin{proof}
To begin with, we show that the product of two tempered distributions maps continuously~$H^{-\eta}\times H^{1-\delta}$ into $H^{-\eta-\delta}$.
Indeed, let us take two tempered distributions $a\in H^{-\eta}$ and~$w\in H^{1-\delta}$, and let us write
\begin{equation} \label{eq:paraprod_a-w}
a\,w\,=\,T_aw\,+\,T_wa\,+\,R(a,w)\,.
\end{equation}
By a systematic use of Proposition \ref{p:op} and of embeddings $H^s\hra B^{s-1}_{\infty,\infty}$, we deduce that
$$
\left\|T_aw+T_wa\right\|_{H^{-\eta-\delta}}\,+\,\left\|R(a,w)\right\|_{B^{1-\eta-\delta}_{1,1}}\,\leq\,C\,\|a\|_{H^{-\eta}}\,\|w\|_{H^{1-\delta}}\,.
$$
At this point, the continuous embedding $B^{1-\eta-\delta}_{1,1}\,\hra\,H^{-\eta-\delta}$ completes the proof of our claim.

For the second point, we apply Bony's paraproduct decomposition to get once again \eqref{eq:paraprod_a-w}, with  $a\in H^{-\eta}$ and $w\in H^{1}$.
%$$
%a\,w\,=\,T_{a}w\,+\,T_{w}a\,+\,R(a,w)\,.
%$$
Keeping in mind the embedding $H^s\,\hookrightarrow\,B^{s-1}_{\infty,\infty}$, by application of Proposition \ref{p:op} we get, for all $\delta>0$,
$$
\left\|T_{a}w\right\|_{H^{-\eta}}\,+\,\left\|T_wa\right\|_{H^{-\eta-\delta}}\,\leq\,C\,\left\|a\right\|_{H^{-\eta}}\,
\|w\|_{H^1}\,.
$$
On the other hand, the remainder operator can be controlled again by Proposition \ref{p:op} in the following way:
$$
\left\|R(a,w)\right\|_{B^{1-\eta}_{1,1}}\,\leq\,C\,\left\|a\right\|_{H^{-\eta}}\,\|w\|_{H^1}\,,
$$
and the continuous embedding $B^{1-\eta}_{1,1}\,\hookrightarrow\,H^{-\eta}$ completes the proof.

As for point (iii) of the statement, for $a\in H^{\eta}$ and $w\in H^{1}$, we write for~$\kappa>0$
$$
 \left\|T_{a}w\right\|_{L^2}\,+\,\left\|T_{w}a\right\|_{L^2}  \leq   \|a\|_{B^{-\kappa}_{\infty,\infty}}\,\|w\|_{H^\kappa}\,+\,\|a\|_{H^1}\,\|w\|_{B^{-1}_{\infty,\infty}}\,,
$$
so in particular taking $\eta =1- \kappa$
$$
 \left\|T_{a}w\right\|_{L^2}\,+\,\left\|T_{w}a\right\|_{L^2}  \leq   \|a\|_{H^{\eta} }\,\|w\|_{H^1}\,.
$$
On the other hand
$$
\left\|R(a,w)\right\|_{L^2} \leq   \|a\|_{H^{\eta} }\,\|w\|_{H^1}\,.
$$

The fourth item of the claim easily follows from analogous arguments, keeping in mind the embedding $H^s\,\hookrightarrow\,L^\infty$ for all $s>1$ (in dimension $d=2$). The last item is also a straightforward consequence of Proposition \ref{p:op}, its proof is omitted.
The corollary is now proved.
\end{proof}

We recall also a classical commutator estimate (see e.g. Lemma 2.97 of \cite{B-C-D}), which is needed in our analysis.
\begin{lemma} \label{l:commut}
Let $h\in\mc{C}^1(\R^d)$ such that $\bigl(1+|\,\cdot\,|\bigr)\what{h}\,\in\,L^1$. There exists a constant $C$ such that,
for any Lipschitz function $\ell\in W^{1,\infty}(\R^d)$ and any $f\in L^p(\R^d)$ and for all $\lambda>0$, one has
$$
\left\|\bigl[h(\lambda^{-1}D),\ell\bigr]f\right\|_{L^p}\,\leq\,C\,\lambda^{-1}\,\left\|\nabla\ell\right\|_{L^\infty}\,\|f\|_{L^p}\,.
$$
\end{lemma}
Going along the lines of the proof, it is easy to see that the constant $C$ depends just on the~$L^1$ norm
of the function $|x|\,k(x)$, where $k\,=\,\mc{F}_\xi^{-1}h$ is the inverse Fourier transform of $h$.

To conclude, let us recall Gagliardo-Nirenberg inequalities, which we have repeatedly used in Subsection \ref{ss:hom_limit}. We refer e.g. to Corollary 1.2 of \cite{C-D-G-G}
for their proof.
\begin{prop} \label{p:Gagl-Nir}
Let $p\in[2,+\infty[\,$ such that $1/p\,>\,1/2\,-\,1/d$. There exists a constant $C>0$ such that, for any domain $\Omega\subset\R^d$
and for all $u\in H^1_0(\Omega)$, the following inequality holds true:
$$
\|u\|_{L^p(\Omega)}\,\leq\,C\,\|u\|_{L^2(\Omega)}^{1-\lam}\;\|\nabla u\|_{L^2(\Omega)}^{\lam}\,,\qquad\qquad\mbox{ with }\qquad
\lam\,=\,\frac{d\,(p-2)}{2\,p}\,\cdotp
$$
\end{prop}

%%%%%%%%%%%%%%%%%%%%%%%%%%%%%%%%%%%%%%%%%%%%%%%%%%%%%%%%%%%%%%%%%%%%%%%%%%%%%%%%%%%%%%%%%%%%
%%%%%%%%%%%%%%%%%%%%%%%%%%%%%%%%%%%%%%%%%%%%%%%%%%%%%%%%%%%%%%%%%%%%%%%%%%%%%%%%%%%%%%%%%%%%

\end{document}